\documentclass{article}
\usepackage{amsmath,amsthm,amssymb}
\usepackage{graphicx}
\usepackage{fullpage}
\newtheorem{thm}{Theorem}[section]
\newtheorem{lemma}{Lemma}[section]
\numberwithin{equation}{section}

\usepackage{listings}
\usepackage{xcolor}
\newcommand{\Real}{\mathbb{R}}

\usepackage{hyperref}
\hypersetup{
	colorlinks=true,  
	linkcolor=blue,   
	citecolor=green,  
	urlcolor=cyan     
}

\allowdisplaybreaks

\title{The Cayley-Moser problem with Poissonian arrival of offers}

\author{Guy Katriel\thanks{katriel@braude.ac.il}\\
	Department of Applied Mathematics,\\
	Braude College of Engineering, Karmiel, Israel}
\date{}

\begin{document}

\maketitle 

\begin{abstract}
We study a version of the classical Cayley-Moser optimal stopping problem, in which
a seller must sell an asset by a given deadline, with the
offers, which are independent random variables with a known distribution, arriving at random times, as a Poisson process.
This continuous-time 
formulation of the problem is much more
analytically tractable than the 
analogous discrete-time problem which 
is usually considered, leading to a simple differential equation that can be explicitly solved  to find the optimal policy.
We study the performance
of this optimal policy, and obtain explicit expressions for the distribution of the realized sale price, as well as for the distribution of the
stopping time. The general results 
are used to explore characteristics of the optimal policy and of the resulting bidding process, and are illustrated by application to 
several specific instances of the offer distribution.
\end{abstract}


	

\section{Introduction}

The Cayley-Moser problem \cite{Cayley,Moser} is a well-known 
optimal stopping problem, in which a seller 
must sell an asset by a given deadline. Potential buyers' price offers arrive sequentially, and are independent and identically distributed (iid) according to a known distribution (which will be called the `offer distribution'). The seller must decide whether to 
accept or reject each offer, with acceptance leading to sale at the offered price and the end of the process. Rejected offers cannot be recalled, and if the deadline is reached then the seller must accept some `salvage value'. The seller's aim is to choose a policy - a rule for deciding whether 
or not to accept an offer, depending on its 
value and on the time remaining until the deadline, 
so as to maximize the expected sale price. The problem thus addresses, in a simple context, the fundamental tradeoff between the potential gain and the risk from waiting.

The standard version of this problem is a `discrete-time' one, in which the number of offers until the deadline is a given integer $n$ \cite{Bauerle,DeGroot,Ferguson,Gilbert,Guttman,Hayes,Moser}. Analysis of this problem leads to an optimal policy defined by an increasing sequence of threshold values (reservation prices) $\mu_n$ ($n=1,2,3,\cdots$), so that an offer $x$ arriving when
the number of remaining offers is $n\geq 1$ is  
accepted if $x\geq\mu_n$. The sequence $\{\mu_n\}$ is defined by 
the nonlinear recursion,
\begin{equation}\label{discrete}\mu_0=0,\;\;\;\mu_n=E[\max(X,\mu_{n-1})],\;\;n\geq 1,\end{equation}
where $X$ is a random variable distributed according to the offer distribution,
which implies that, although the values can be computed $\mu_n$ iteratively,  an explicit expression for these values is not 
available. We note that the value $\mu_n$ also represents the expectation of the realized sale price, when the maximal number of offers is $n$ and the optimal policy is followed.

In recent years considerable attention has been given to the Cayley-Moser problem, and generalizations of it, in the context of {\it{prophet inequalities}}, which bound the expectation of the 
realized sale price under the optimal policy from below, in terms of the expected price achieved by a `prophet' who can observe all offers in advance, and choose the maximal one. Relatedly, there has also been interest in obtaining simple policies, e.g. with only a single threshold rather than a time-varying one, which are sub-optimal but which can be proved to attain results approximating those of the optimal policy up to a factor (see \cite{Allaart,Correa,Hill,Livanos,Lucier} and references therein). In the present work, however, we are focused on the precise characterization of the optimal policy, and of its performance measures.

Here we study a continuous-time version
of the Cayley-Moser problem, in which the (iid) offers 
arrive as a Poisson process with given rate 
$\lambda$, and the time until the deadline is given by a positive real number $t>0$. 
The sought-after policy, which maximizes the expectation of the sale price, is defined by
a threshold function $\mu(t)$, so that an offer $x$ made when the remaining time to the deadline is $t$ will be accepted if $x\geq \mu(t)$.

It may be argued that this Poissonian model with random arrival times is more `realistic' than the discrete-time model which assumes a fixed and known number of offers, but the real motivation underlying this work 
is the observation that this model is in fact much more analytically tractable then the 
discrete-time model, allowing us to obtain a range of explicit results which cannot be obtained for the discrete-time model.

As already shown in \cite{Karlin} (see also \cite{Allaart,David,Elfving,Sakaguichi} for related results), in the Poissonian case the optimal policy $\mu(t)$ can be obtained in an explicit form, which 
is unattainable in the discrete case. This stems from the fact that while the discrete model 
leads to a sequence of thresholds defined by the
nonlinear recursion \eqref{discrete}, which cannot be solved in closed form, the Poissonian model leads to a 
a simple differential equation for the optimal
policy $\mu(t)$, which can be solved by a direct integration (see subsection \ref{optimal_policy}). As an example, when
the offers are uniformly distributed on $[0,1]$,
the discrete model leads, using \eqref{discrete}, to the sequence defined by the recursion
\begin{equation}\label{disc}\mu_0=0,\;\;\;\mu_n=\frac{1}{2}(1+\mu_{n-1}^2),\;\;n\geq 1,
\end{equation}
for which an explicit expression is not available (even the study of the precise  asymptotic behavior of the above-defined sequence as $n\rightarrow \infty$ is a delicate question \cite{entwistle,fitch}). 
By contrast, for the corresponding Poissonian model we obtain an 
{\it{explicit}} expression for the optimal policy, given by
(see subsection \ref{policy_uniform})
\begin{equation}\label{ex0}\mu(t)=1-\frac{2}{\lambda t+4}.
\end{equation}

In this work we will show, furthermore, that in the Poissonian model it is possible to obtain detailed and explicit 
results regarding the {\it{performance}} of the optimal policy, which have not been obtained in previous works.

The outcome of 
applying the optimal policy, starting at a time when the deadline is $t$ time units in the future, is a random variable $S_t$ - the {\it{realized sale price}}. We will see that the
expectation of this random variable is $E[S_t]=\mu(t)$ - so that the function $\mu(t)$ plays a dual role, both as the optimal policy and as the expected gain.
It is also important, however, to evaluate the {\it{variance}} of $S_t$, which will indicate the range of likely sale prices under the optimal policy.
In fact we will be able to obtain a formula for the full {\it{probability density}} of $S_t$. Here, again, derivation of
a differential equation and its solution, this time a first order linear differential equation, will play a key role.
The probability density takes an interesting form,
with different analytic expressions in each of three intervals. Such explicit results on the probability distribution of the realized sale price cannot be attained for the discrete-time version of the model.

Another important performance metric to evaluate  is the {\it{time to sale}}, that is the stopping time: assuming the deadline for performing the sale is initially $t$ time units in the future, how much of this available duration will be exploited, when the optimal policy is followed? 
We will obtain an explicit formula for the distribution of the random variable $T_t$, the time to sale, assuming the optimal policy is followed starting when the time to the deadline is $t$. Consequently, we will also obtain expressions for the expectation and the variance of $T_t$. In previous literature \cite{entwistle,Mazalov}, {\it{asymptotic}} results for the 
expectation and variance of the stopping time have been obtained
for the discrete-time version of the problem. We will show that our exact 
formulas for the continuous-time case have
the same asymptotics.

All the results described above will be 
illustrated by application to several specific offer 
distributions. They can be applied to any desired offer distribution, although it may be necessary to evaluate some integrals numerically.
We also carry out some numerical simulations of the bidding process, and demonstrate the perfect fit of our analytical results with the statistics collected from the simulations.

This work is divided into four subsequent sections. In section \ref{optimal_policy0} we derive 
the optimal policy function $\mu(t)$, and use this result to explore general properties of this function, giving a
characterization of functions that can arise as optimal policies, as well as bounds on this function in terms of moments of the offer distribution.
In section \ref{price_distribution}
we obtain the full distribution of the realized sale price $S_t$, and compute the variance of this distribution. In
section \ref{time_to_sale} we  obtain an explicit expression for
the distribution of the time to sale 
$T_t$, and use it to compute the expectation and variance of $T_t$.
In section \ref{asy} we 
use the results obtained in the
previous sections to obtain the asymptotics of the expectation and variance of the sale price, as $t\rightarrow \infty$, under 
appropriate assumptions on the tail
behavior of the offer distribution. We also study the the limiting distribution of $\frac{1}{t}T_t$ as 
$t\rightarrow \infty$. 

More detailed overviews of the contents are given at the beginning of each section.

\section{The optimal policy}
\label{optimal_policy0}

After defining the bidding process
with Poissonian offer arrivals in subsection \ref{setup}, in subsection \ref{optimal_policy} we will derive the 
expression for the optimal policy $\mu(t)$, a result which has already 
been derived in \cite{Karlin}, though we 
use a different approach to obtain it, which will also serve us in the derivation of the distribution of the realized sale price in section \ref{price_distribution}. In subsection \ref{mu_examples} we apply the general  result to determine explicit expressions for optimal policies for some specific examples of 
offer distributions.
In subsection \ref{characterization}
we will derive some general properties 
of the optimal policy function $\mu(t)$,
and prove that these properties completely characterize the class of functions which may arise as optimal policies. In subsection \ref{moments} we show that under certain restrictions on the right tail of the offer distribution, we can obtain simple bounds on the growth of the optimal policy function $\mu(t)$.

\subsection{Problem setup and notation}
\label{setup}

The inputs to the problem are:

\begin{itemize}
	
	\item The {\it{marketing period}}: a time duration $t$ 
	during which the sale must be completed.
	It should be kept in mind that 
	as time proceeds, $t$ {\it{decreases}}, with $t=0$ corresponding to the deadline.
	
	\item The rate $\lambda$ of arrival of offers. Offers arrive as a Poisson process, so that 
	the distribution of the time durations $D$ between 
	consecutive offers is given by
	\begin{equation}\label{pp}P(D\leq d)=1-e^{-\lambda d}.
	\end{equation}
	
	\item The {\it{offer distribution}}  from which 
	the iid price offers are to be drawn. This distribution will 
	be characterized by a cumulative distribution function (CDF) $F:\Real\rightarrow [0,1]$. We do not restrict the offers to be non-negative. We will use $X$ to denote a generic random variable which
	is distributed according to $F$:
	\begin{equation}\label{xx}P(X\leq x)=F(x).\end{equation}
	Denoting $X_+=max(X,0)$, it will be assumed that 
	\begin{equation}\label{fm}
		E[X_+]=\int_{0}^\infty xdF(x)<\infty.\end{equation}
	If the support of the offer distribution is bounded from above, we will denote the maximum of its support by
	\begin{equation}\label{def_M}M=\sup\; \{ x\;|\; F(x)<1\},\end{equation}
	with $M=+\infty$ if the support is unbounded.
	If $F$ is absolutely continuous, we will denote 
	the corresponding probability density by $f(x)$.
	
	\item The {\it{residual price distribution}} (or `salvage value'):
	If the deadline is reached before an offer has been accepted, the asset will be sold at a price drawn from this distribution. We denote the corresponding cumulative distribution function by 
	$F_0:\Real\rightarrow [0,1]$, and a generic random variable with this distribution by $X_0$. We will assume
	\begin{equation}\label{fm0} E[|X_0|]=\int_{-\infty}^\infty |x|dF_0(x)<\infty.
	\end{equation}
	Defining 
	\begin{equation}\label{def_mu0}\mu_0\doteq E[X_0]=\int_{-\infty}^\infty xdF_0(x),
	\end{equation}
	we will also assume that
	\begin{equation}\label{mM}0\leq \mu_0<M,\end{equation}
	where $M$ is given by $\eqref{def_M}$. 
	The purpose of the assumption $\mu_0<M$ is 
	to avoid the trivial situation
	$\mu_0\geq M$ in which any offer one can receive during the marketing period is no higher than the expected residual price, so that to maximize the expected sale price one should always wait until the deadline.
	The assumption that $\mu_0\geq 0$ makes the formulation of some results more convenient, and it represents no real loss of generality, since if $E[X_0]<0$ one can 
	always add a constant to both $X_0$ and $X$,
	to obtain a problem which is equivalent, and which satisfies $\mu_0\geq 0$. Note that \eqref{mM} implies $M>0$, so that not all potential offers are negative.
\end{itemize}

During the bidding process, offers arrive as
a Poisson process with inter-arrival durations
$D_i$ ($i=1,2,..$), so that the time $t_i$ remaining when the $i$-th offer arrives is 
\begin{equation}\label{ti}t_i=(t_{i-1}-D_i)_+=\max(t_{i-1}-D_i,0),\;\;\;t_0=t,\end{equation}
where $D_i$ are iid and distributed according to \eqref{pp}. The size of the $i$-th offer is $X_i$, iid random variables distributed according to \eqref{xx}. The seller must decide 
whether to accept the offer $X_i$, in which case
the bidding process ends with payoff (realized sale price) $S_t=X_i$ and stopping time (time to sale) $T_t=t-t_i$, or to wait for the next offer, in which case the current offer cannot be recalled. The process also ends
if $t_i$ given by \eqref{ti} satisfies $t_i= 0$, so that the deadline is reached,  in which case the payoff is a random value $S_t=X_0$ distributed according to the residual price distribution $F_0$ and the stopping time is $T_t=t$.

Regarding the residual price distribution, we can envisage different scenarios depending on the nature of the asset to be sold. If the asset is ephemeral, {\it{e.g.}} a concert ticket or perishable food, then
its value drops to $0$ after the deadline (the concert date or expiration date), so that $X_0=0$, hence $F_0(x)=1$ for all $x>0$. Another possibility, which will be
used in our examples, is that $F_0=F$, which can be interpreted to mean that if the deadline is reached before an offer has been accepted then the seller waits until the arrival of the next offer, and accepts it unconditionally.

We have thus defined two important 
random variables, parameterized by the marketing period $t$, associated with the bidding process, and assuming the optimal policy is followed: the realized sale price $S_t$ and the time to sale $T_t$. After determining the optimal policy, we will study the distribution of these two random variables.

\subsection{Determination of the optimal policy}
\label{optimal_policy}

We define the value function $V(t,x)$:
assuming an offer of $x$ arrives when the time to 
the deadline is $t$, $V(t,x)$ is the maximum
expected value that can be achieved. If the offer 
is accepted then $V(t,x)=x$, while if
it is rejected then this value will be 
\begin{equation}\label{dmu}\mu(t)\doteq E[V((t-D)_+,X)],\end{equation}
where $D$ is the duration to the 
next offer, which is exponentially distributed 
with rate $\lambda$, and $X$ is distributed according to the offer distribution, with $X,D$ independent. We thus have
the Bellman equation
\begin{equation}\label{Bellman}V(t,x)=\max(x,E[V((t-D)_+,X)])=\max(x,\mu(t)),\end{equation}
together with the boundary condition:
\begin{equation}\label{bc}V(0,x)=\mu_0,\end{equation}
where $\mu_0=E[X_0]$ is the expectation of the 
residual distribution $F_0$.

By \eqref{Bellman}, an offer of $x$ at time $t$ to the deadline should be accepted if $x\geq \mu(t)$ (in fact if $x=\mu(t)$ one is indifferent between acceptance and rejection, but since this is an event of probability $0$ none of our result will be affected by the decision in such a case). Thus $\mu(t)$ provides us with the threshold defining the optimal policy.

To find the function $\mu(t)$, we first use \eqref{dmu},\eqref{bc} to write
\begin{align}\label{c1}\mu(t)&=E[V((t-D)_+,X)]=P(D\geq t)\mu_0+ \int_0^t \lambda e^{-\lambda (t-s)} E[V(s,X)]ds\nonumber\\
	&=e^{-\lambda t}\mu_0+ e^{-\lambda t}\int_0^t \lambda e^{\lambda s} \left(\int_{-\infty}^\infty V(s,u)dF(u) \right)ds.\end{align}
Now, using \eqref{Bellman}, we compute
\begin{align}\label{c2}\int_{-\infty}^\infty V(s,u)dF(u)&=\int_{-\infty}^\infty \max(u,\mu(s))dF(u)\nonumber\\
	&=\mu(s)\int_{-\infty}^{\mu(s)}dF(u)+\int_{\mu(s)}^\infty udF(u)
	=\mu(s)F(\mu(s))+\sigma(\mu(s)),
\end{align}
where in the last term we use the function defined 
by
\begin{equation}\label{def_sigma}
	\sigma(x)=\int_x^\infty udF(u).
\end{equation}
Combining \eqref{c1} and \eqref{c2}, we have
the following Volterra-type integral equation:
\begin{equation}\label{int}\mu(t)=e^{-\lambda t}\mu_0+ e^{-\lambda t}\int_0^t \lambda e^{\lambda s} \left[\mu(s)F(\mu(s))+\sigma(\mu(s)) \right]ds.\end{equation}
Note that by setting $t=0$ we obtain that
\begin{equation}\label{init0}
	\mu(0)=\mu_0.
\end{equation}
To solve \eqref{int}, we multiply both sides by
$e^{\lambda t}$ and then differentiate both sides, leading to
$$e^{\lambda t}[\mu'(t)+\lambda \mu(t)]=\lambda e^{\lambda t}\left[\sigma(\mu(t))+F(\mu(t))\mu(t)  \right],$$
which we can write as 
\begin{equation}\label{de1}\mu'(t)=\lambda \sigma(\mu(t))-\lambda [1-F(\mu(t))]\mu(t)  .\end{equation}
Using an integration by parts, we have
\begin{eqnarray}\label{calc0}
	\sigma(x)-(1-F(x))x&=&\int_x^\infty udF(u)-x\int_x^\infty dF(u)=\int_x^\infty (u-x)dF(u)=
	-\int_x^\infty (u-x)d[1-F(u)]\nonumber\\
	&=&-(u-x)[1-F(u)]\Big|_{u=x}^{u=\infty}
	+\int_x^\infty [1-F(u)]du=\int_x^\infty [1-F(u)]du,
\end{eqnarray}
so that, defining the function
\begin{equation} \label{defphi}\phi(x)=\int_x^\infty [1-F(u)]du,
\end{equation}
we have obtained the following
\begin{thm}\label{optimal}
	The optimal policy $\mu(t)$ is given by the solution of the initial value problem
	\begin{equation}\label{ode}
		\begin{cases}
			\mu'(t)=\lambda \phi(\mu(t))\\
			\mu(0)=\mu_0,
		\end{cases}
	\end{equation}
	where $\phi$ is defined by \eqref{defphi}.
	\eqref{ode} can be integrated to obtain the 
	explicit expression
	\begin{equation}\label{ef}\mu(t)=\Psi^{-1}(\lambda t),\end{equation}
	where $\Psi^{-1}:[0,\infty)\rightarrow [\mu_0,M)$ is the inverse of the function
	\begin{equation}\label{def_Psi}\Psi(x)\doteq\int_{\mu_0}^x \frac{du}{\phi(u)},\;\;\;\; x\in [\mu_0,M).\end{equation}
\end{thm}
The function $\phi$ is of central importance throughout this work, since, as will be seen, all the explicit formulas for various quantities that will be obtained will involve this function.
We note that this function can also 
be given a probabilistic expression:
\begin{equation}\label{prob_phi}\phi(x)=E[(X-x)_+].\end{equation}
In case $F$ is absolutely continuous, so that a density $f=F'$ exists, we have an alternative expression for the function $\phi$, given by (see \eqref{calc0})
\begin{equation}\label{defphi1}
	\phi(x)=\int_x^\infty (u-x)f(u)du=\int_0^\infty uf(x+u)du.
\end{equation}

As a final important remark, note that, 
denoting by $S_t$ the random variable describing the realized sale price assuming the bidding process starts 
$t$ time units prior to the deadline and the optimal policy is employed, we have
\begin{equation}\label{Es}E[S_t]=E[V((t-D)_+,X)]=\mu(t)\end{equation}
($D$ is the time we wait for the first offer, and $X$ its value), hence the function $\mu(t)$, which we now know how to compute, describes both the threshold value for accepting an offer and the expected sale price, when the time to the deadline is $t$.

\subsection{Examples}
\label{mu_examples}

In the following we apply the result of 
Theorem \ref{optimal} to compute the 
optimal policy
for several examples of offer distributions $F$. Note that, since the
residual distribution $F_0$ 
influences the function $\mu(t)$ 
only through its expectation $\mu_0$, there is no need to assume a specific form for this distribution.

The explicit expressions for $\mu(t)$ 
allow us to easily obtain the asymptotic
behavior of $\mu(t)$ as $t\rightarrow +\infty$, and we compare these results 
with those obtained in \cite{entwistle}
in the discrete-time case, where explicit expressions for the optimal policy $\mu_n$ are not available, and show that they agree. A more general discuss of asymptotics of $\mu(t)$ is made in subsection \ref{asymptotics1} below.

\subsubsection{Uniform offer distribution}
\label{policy_uniform}

We assume the offers are uniformly 
distributed on $[a,b]$, so that 
their probability density is 
\begin{equation}\label{uniform_denisity}f(x)=\frac{1}{b-a}\chi_{[a,b]}(x),\end{equation}
where $\chi_{[a,b]}$ denotes the 
characteristic function of the interval 
$[a,b]$. The residual distribution
$F_0(x)$ can be arbitrary, but to avoid the trivial case in which any offer made before the deadline will be rejected, we assume  $\mu_0=E(X_0)\in [0,b)$ (see \eqref{mM}).

We then have, using \eqref{defphi1}, 
\begin{align}\label{phi_uniform}
	\phi(x)&=\frac{1}{b-a}\int_0^\infty u \chi_{[a,b]}(u+x)du=\frac{1}{b-a}\begin{cases}
		\int_{a-x}^{b-x}udu	&  x<a\\
		\int_{0}^{b-x} u du  &  a\leq x<b\\
		0	  &    x\geq b
	\end{cases}
		=\begin{cases}
			\frac{a+b}{2}-x	&  x<a\\
			\frac{(b-x)^2}{2(b-a)}  &  a\leq x<b\\
			0	  &    v\geq b
		\end{cases}
	\end{align}
	Because of the piecewise nature of the function $\phi$, we now need to 
	consider two cases:
	
	(1) Assuming first that $\mu_0\in [a,b]$, we have, using \eqref{def_Psi}, for $\mu_0\leq x< M=b$,
	\begin{align*}\Psi(x)&=\int_{\mu_0}^x \frac{dv}{\phi(v)}
		=2(b-a)\int_{\mu_0}^x \frac{dv}{(b-v)^2}
		=2(b-a)\left[\frac{1}{b-x}-\frac{1}{b-\mu_0} \right],
	\end{align*}
	so that, by \eqref{ef},
	\begin{equation}\label{mu_uniform_g}
		\mu(t)=\Psi^{-1}(\lambda t)=b-\frac{2(b-a)}{\lambda t+2\cdot \frac{b-a}{b-\mu_0}}.
	\end{equation}
	For example, in the case of a uniform distribution on $[0,1]$ ($a=0,b=1$), and assuming that reaching the deadline
	before making a sale leads to loss of the sale, we should take $\mu_0=0$ in the above expression, giving
	$$\mu(t)=\frac{\lambda t}{\lambda t+2},$$
	while in the case 
	that one accepts the first offer 
	following the deadline, we take $\mu_0=\frac{a+b}{2}=\frac{1}{2}$,
	giving \eqref{ex0}.
	
	(2) Assuming now that $\mu_0\in [0,a)$, we need to modify our calculation. For 
	$x\in [\mu_0,a]$ we have
	$$\Psi(x)=\int_{\mu_0}^x \frac{dv}{\phi(v)}=\int_{\mu_0}^x \frac{dv}{\frac{a+b}{2}-v}dv=\ln\left(\frac{a+b-2\mu_0}{a+b-2x} \right),$$
	while for $x\in [a,b)$
	$$\Psi(x)=\int_{\mu_0}^a \frac{dv}{\phi(v)}+\int_{a}^x \frac{dv}{\phi(v)}=\int_{\mu_0}^a \frac{dv}{\frac{a+b}{2}-v}dv+2(b-a)\int_{a}^x \frac{dv}{(b-v)^2}=\ln\left(\frac{a+b-2\mu_0}{b-a} \right)+2\cdot\frac{x-a}{b-x}.$$
	Defining
	$$t^*=\Psi(a)=\frac{1}{\lambda}\ln\left(\frac{a+b-2\mu_0}{b-a} \right),$$
	the above expressions lead to
	\begin{equation}\label{mu_uniform_g1}\mu(t)=\Psi^{-1}(\lambda t)=\begin{cases}
			\frac{1}{2}(a+b-(a+b-2\mu_0)e^{-\lambda t}) & 0\leq t\leq t^*\\
			b-\frac{2(b-a)}{\lambda (t-t^*)+2}   & t>t^*.
	\end{cases}\end{equation}
	Note that when $t<t^*$ we have 
	$\mu(t)<a$, and since offers are in the range $[a,b]$, this means that when the time to the deadline is less than $t^*$, all offers will be accepted. The precise 
	value of $\mu(t)$ is thus not important
	as a threshold when $t\leq t^*$, but it does provide the value $E[S_t]$.

	In terms of asymptotic behavior, \eqref{mu_uniform_g},\eqref{mu_uniform_g1} imply
	\begin{equation}\label{ba0}\mu(t)=b-\frac{2(b-a)}{\lambda t}+O\left(\frac{1}{t^2}\right)\;\;\;{\mbox{ as }} t\rightarrow +\infty,\end{equation}
	which is analogous to the asymptotic 
	formula obtained in \cite{entwistle} (eq. 20, replacing $n$ by $\lambda t$) for the discrete-time version of the problem, in which case a closed formula for the 
	threshold values is not attainable.
	
	
	
	
	\subsubsection{Exponential offer distribution}
	\label{policy_exponential}
	
	We assume now that the offers are exponentially distributed with mean $\eta$:
	\begin{equation}\label{exp}f(x)=\frac{1}{\eta}e^{-\frac{x}{\eta}},\;\;\;F(x)=1-e^{-\frac{x}{\eta}},\;\;\;x\geq 0.
	\end{equation}
	By \eqref{defphi} we have
	\begin{equation}\label{phi_exp}\phi(x)=\int_x^\infty e^{-\frac{u}{\eta}}du=\eta e^{-\frac{x}{\eta}},\end{equation}
	by \eqref{def_Psi},
	$$\Psi(x)=\int_{\mu_0}^x \frac{du}{\phi(u)}=\int_{\mu_0}^x \frac{1}{\eta}e^{\frac{u}{\eta}}du=e^{\frac{x}{\eta}}-e^{\frac{\mu_0}{\eta}},$$
	so, by \eqref{ef},
	\begin{equation}\label{mu_exp}\mu(t)=\Psi^{-1}(\lambda t)=\eta \ln\left(\lambda t+e^{\frac{\mu_0}{\eta}} \right).\end{equation}
	\eqref{mu_exp} agrees
	with the large $t$ asymptotic expression obtained in \cite{entwistle} (eq.14) for the 
	discrete-time version of the problem.
	
	\subsubsection{Pareto offer distribution}
	\label{policy_pareto}
	
	We consider the Pareto distribution (with $\alpha>1$)
	\begin{align}\label{pareto}f(x)=\begin{cases}
			0 &  x < x_m\\
			\frac{\alpha x_m^\alpha}{x^{\alpha+1}} & x\geq x_m
		\end{cases},\;\;\;F(x)=\begin{cases}
			0 &  x < x_m\\
			1-\left(\frac{x_m}{x} \right)^\alpha & x\geq x_m
		\end{cases}.\end{align}
	with expectation
	\begin{equation}\label{mu_0_pareto}E[X]=\frac{\alpha x_m}{\alpha-1}.\end{equation}
	The residual distribution is arbitrary, but to somewhat simplify calculations we will assume now that 
	$$\mu_0=E[X_0]\geq x_m.$$
	
	Using \eqref{defphi} we obtain, for 
	$x\geq x_m$,
	\begin{equation}\label{phi_pareto}
		\phi(x)=\int_x^\infty \left(\frac{x_m}{u} \right)^\alpha du=\frac{x_m^\alpha}{\alpha-1}\cdot \frac{1}{x^{\alpha-1}}.\end{equation}
	From \eqref{def_Psi} we obtain, for $x\geq \mu_0$,
	$$\Psi(x)=\frac{\alpha-1}{x_m^\alpha}\int_{\mu_0}^x v^{\alpha-1}dv=\frac{\alpha-1}{x_m^\alpha}\cdot\frac{1}{\alpha }\left[x^\alpha-\mu_0^\alpha \right]\;\;\;\Rightarrow\;\;\; \Psi^{-1}(w)=\left[\alpha\left(\frac{x_m^\alpha}{\alpha-1}\right) w+\mu_0^\alpha\right]^{\frac{1}{\alpha}},\;\;w\geq 0$$
	so \eqref{ef} gives
	\begin{equation}\label{mu_pareto_g}
		\mu(t)=\Psi^{-1}(\lambda t)=\left[\left(\frac{\alpha x_m^\alpha}{\alpha-1}\right) \lambda t+\mu_0^\alpha\right]^{\frac{1}{\alpha}}=\mu_0[c\lambda t+1]^{\frac{1}{\alpha}},\end{equation}
	where we have set
	\begin{equation}\label{defc}c=\frac{\alpha}{\alpha-1}\left(\frac{x_m}{\mu_0} \right)^\alpha.\end{equation}
	\eqref{mu_pareto_g} implies
	$$\mu(t)=x_m\left(\frac{\alpha }{\alpha-1}\cdot \lambda  t \right)^{\frac{1}{\alpha}}+O\left(\frac{1}{t^{1-\frac{1}{\alpha}}} \right)\;\;\;{\mbox{ as }} t\rightarrow +\infty,$$
	which is analogous to the large time asymptotics obtained in \cite{entwistle}
	(eq. 17) for the discrete-time version of the problem.

	\label{relations}

	\subsection{Characterization of optimal policy functions}
	\label{characterization}
	
	Since the function $\mu(t)$ plays two important roles, representing both the optimal sale policy and the expected value $E[S_t]$ of the realized sale price under this policy, it is 
	of interest to study its properties. For any given 
	offer distribution we can compute $\mu(t)$ using Theorem \ref{optimal} (as we have done in the examples in subsection \ref{mu_examples}), up to the possible
	need to compute an integral numerically, but we are also interested in deriving general characteristics of this function which are not dependent on the particular offer distribution, both those which are shared by the policy function for any choice of the distributions $F$,$F_0$, and those which are
	valid for certain classes of offer distributions.

	In the next theorem we derive a series 
	of properties that any optimal policy function $\mu(t)$ must possess. Moreover,
	the second part of the theorem shows
	that any function $\mu(t)$ which satisfies this list of properties is an optimal policy for some choice of the offer and residual distributions. Thus the properties listed in fact characterize the optimal policy functions, and any additional properties of optimal policy functions, which do not 
	follow from those in the list, must therefore depend on specific assumptions about the offer distribution. In section \ref{relations}, we will restrict 
	the offer distribution in several ways, to obtain further properties of $\mu(t)$. A natural question for future research is whether in the discrete-time case it is possible to obtain a result analogous to Theorem \ref{char}, characterizing the sequences $\{\mu_n\}$ which can be threshold sequences corresponding to some offer distribution.
	
	\begin{thm}\label{char}
		(1) Let $F,F_0$ be a pair of distributions satisfying \eqref{fm},\eqref{fm0},\eqref{mM}, and $\lambda>0$.
		Then the corresponding optimal policy $\mu: [0,\infty)\rightarrow [0,\infty)$ has the following properties:
		\begin{itemize}
			\item[(i)] $\mu(0)=\mu_0=\int_{-\infty}^\infty xdF_0(x)$ and  $\lim_{t\rightarrow +\infty}\mu(t)=M$, where $M$ is defined by
			\eqref{def_M}.
			
			\item[(ii)] $\mu$ is continuously differentiable.
			
			\item[(iii)] $\mu'$ is absolutely continuous and satisfies $\mu'(t)>0$ for all $t\geq 0$. In particular $\mu$ is strictly monotone increasing.
			
			\item[(iv)] $\mu(t)$ is strictly concave.
			
			\item[(v)] The function \begin{equation}\label{def_h}h(t)=-\frac{\mu''(t)}{\mu'(t)}
			\end{equation}
			(which, by (iii),(iv) is defined and positive {\it{a.e.}} on $[0,\infty)$) is a monotone decreasing function.
			
			\item[(vi)] The function $h$ satisfies
			\begin{equation}\label{h_cond}h(0+)\doteq\lim_{t\rightarrow 0+}h(t)\leq \lambda.
			\end{equation}
		\end{itemize}
		
		(2) Conversely, let $\mu:[0,\infty)\rightarrow [0,\infty)$ be any function
		satisfying (ii),(iii),(iv),(v) and let $\lambda$ satisfy \eqref{h_cond}.  Then, defining
		$M=\lim_{t\rightarrow \infty}\mu(t),$ there 
		exists a distribution $F$ satisfying 
		\eqref{fm}, such that for any distribution $F_0$ satisfying \eqref{fm0} with $E[X_0]=\mu(0)$, we have that 
		$\mu(t)$ is the optimal policy corresponding to $F,F_0$, and $\lambda$.
	\end{thm}
	
	\begin{proof}
		By its definition \eqref{defphi}, the function $\phi$ is positive on $(-\infty,M)$. In addition
		$\phi$ is absolutely continuous  $[0,M)$, with
		$$\phi'(x)=F(x)-1<0,\;\;\;{\mbox{for {\it{a.e.}} }}\;x\in [0,M),$$
		and in particular is strictly decreasing on 
		$(-\infty,M)$.
		We also have (both when $M<\infty$ and when $M=+\infty$)
		\begin{equation}\label{pgz}
			\lim_{x\rightarrow M-}\phi(x)=0.
		\end{equation}
		
		By our assumption \eqref{mM} and the positivity of $\phi$, the function
		$\Psi$ given by \eqref{def_Psi} is well-defined on $[\mu_0,M)$.
		Moreover $\Psi$ is continuously differentiable on $[\mu_0,M)$, with
		\begin{equation}\label{dpsip}\Psi'(x)=\frac{1}{\phi(x)}>0,\end{equation}
		and in particular it is strictly monotone
		increasing on $[\mu_0,M)$. Since $\phi(x)$ is absolutely continuous and positive 
		$\Psi'(x)=\frac{1}{\phi(x)}$ is also absolutely continuous on any closed subinterval of $[\mu_0,M)$.
		If $M<\infty$, then $\phi(M)=0$ and
		\begin{align*}\phi(u)&=\int_{u}^M  [1-F(x)]dx\leq M-u,\;\;\; u<M\\
			&\Rightarrow\;\;\Psi(x)=\int_{\mu_0}^x \frac{du}{\phi(u)}\geq \int_{\mu_0}^x
			\frac{du}{M-u}=\ln\left( \frac{M-\mu_0}{M-x}\right),
		\end{align*}
		which implies 
		\begin{equation}\label{psi_inf}\lim_{x\rightarrow M-}\Psi(x)=+\infty.\end{equation}
		\eqref{psi_inf} also holds if $M=+\infty$,  
		since, in this case, by \eqref{pgz} the integrand $1/\phi(u)$ in \eqref{def_Psi} 
		goes to $+\infty$ as $u\rightarrow\infty$.
		Since $\phi(x)$ is strictly decreasing, 
		$\Psi'(x)=\frac{1}{\phi(x)}$ is strictly increasing, hence $\Psi(x)$ is strictly convex.
		
		Since $\Psi$  satisfies $\Psi(\mu_0)=0$, and by \eqref{dpsip}, and \eqref{psi_inf},
		the function 
		$\Psi^{-1}$ is well defined and strictly increasing on $[0,\infty)$, with
		\begin{equation}\label{psip}\Psi^{-1}(0)=\mu_0,\;\;\; \lim_{v\rightarrow +\infty} \Psi^{-1}(v)=M,\end{equation}
		and is also continuously differentiable.
		Moreover, since 
		$$(\Psi^{-1})'(z)=\frac{1}{\Psi'(\Psi^{-1}(z))}=\phi(\Psi^{-1}(z)),\;\;\;z\in [0,\infty),$$
		we have that $(\Psi^{-1})'(z)$ is positive and is absolutely continuous on every closed sub-interval of $[0,\infty)$, as a composition of the absolutely continuous function $\phi$ and the continuously  differentiable function $\Psi^{-1}$.
		Since $\Psi$ is strictly convex 
		$\Psi^{-1}$ is strictly concave.
		
		Thus, by \eqref{ef} and \eqref{psip}, $\mu(t)$ is defined on $[0,\infty)$, with
		\begin{equation}\label{ll}\mu(0)=\mu_0,\;\;\;\lim_{t\rightarrow +\infty}\mu(t)=M,\end{equation}
		so we have (i).
		Since $\Psi^{-1}$ is continuously differentiable and $(\Psi^{-1})'$ is positive and absolutely continuous, \eqref{ef} implies 
		(ii),(iii). Since $\Psi^{-1}$ is strictly concave, 
		\eqref{ef}, implies the same for $\mu(t)$, so we have (iv). Finally, note that we have,
		for {\it{a.e.}} $t\geq 0$,
		\begin{equation}\label{prop}
			\mu'(t)=\frac{\lambda}{\Psi'(\Psi^{-1}(\lambda t))}=\lambda\phi(\mu(t))\;\;\Rightarrow\;\; \mu''(t)=\lambda \phi'(\mu(t))\mu'(t)\;\;\Rightarrow\;\; 
			h(t)=-\frac{\mu''(t)}{\mu'(t)}=\lambda[1-F(\mu(t))],\end{equation}
		and since $F$, $\mu$ are increasing functions, we conclude that $h$
		is monotone decreasing. Also, by \eqref{prop}, we have
		$h(0+)=\lambda (1-F(\mu_0+))\leq \lambda$, so that \eqref{h_cond} holds.
		
		(2) To prove the converse direction, assume 
		$\mu: [0,\infty)\rightarrow [0,\infty)$  is a function satisfying  (ii),(iii),(iv), and define
		\begin{equation}\label{dd}
			\mu_0=\mu(0),\;\;M=\lim_{t\rightarrow +\infty}\mu(t).
		\end{equation}
		We also assume $\lambda$ satisfies \eqref{h_cond}. 
		We will construct 
		distribution a $F(x)$ satisfying 
		\eqref{fm}, such that, for
		an arbitrary 
		distribution $F_0(x)$  with $\int_{-\infty}^\infty xdF_0(x)dx=\mu_0$,
		$\mu(t)$ is the optimal policy corresponding to $F$ and $F_0$.
		
		By \eqref{dd} and (iii), we have that $\mu_0<M$ and that the function 
		$\mu^{-1}:[\mu_0,M)\rightarrow [0,\infty)$ exists, is strictly monotone increasing, and satisfies
		\begin{equation}\label{mip}\mu^{-1}(\mu_0)=0,\;\;\lim_{x\rightarrow M-}\mu^{-1}(x)=+\infty.\end{equation}
		
		By (ii),(iii),(v) we have that the function $h(t)$ defined by \eqref{def_h} is an {\it{a.e.}} defined, non-negative and monotone decreasing function, hence  the 
		limit 
		$$h(+\infty)=\lim_{t\rightarrow +\infty}h(t)$$
		exists and satisfies $h(+\infty)\geq 0$. 
		We claim that
		\begin{equation}\label{claim}{\mbox{if }}M=+\infty\;\;{\mbox{then}}\;\;h(+\infty)=0.
		\end{equation}
		Indeed, since
		$$\ln\left(\frac{\mu'(t)}{\mu'(0)}\right)=-\int_0^t h(s)ds\leq -\int_0^t h(+\infty)dt=-h(+\infty)t\;\;\Rightarrow\;\; \mu'(t)\leq \mu'(0)e^{-h(+\infty)t},$$
		if $h(+\infty)>0$ we obtain
		$$\mu(t)\leq \mu(0)+\mu'(0)\int_0^t e^{-h(+\infty)s}ds\leq \mu(0)+\mu'(0)\int_0^{\infty} e^{-h(+\infty)s}ds <\infty,$$
		so $M<\infty$, and we have proved the claim \eqref{claim}.
		
		We now define $F(x)$ by
		\begin{equation}\label{def_FF}F(x)=\begin{cases}
				0 & x<\mu_0\\
				1-\frac{1}{\lambda}h(\mu^{-1}(x)) & \mu_0\leq x<M\\
				1 & x\geq M
		\end{cases}\end{equation}
		(if $M=+\infty$ then the last term for $x\geq M$ is superfluous). To show that $F(x)$ is 
		a cumulative distribution function, we have to show it is monotone increasing  
		with $F(-\infty)=0,F(+\infty)=1$.
		Since $h$ is decreasing and $\mu^{-1}$ is increasing, 
		$F$ is a monotone increasing function on
		$[\mu_0,M)$, and since 
		$\lambda$ satisfies \eqref{h_cond}, we have
		\begin{equation*}\label{gi0}F(\mu_0+)=1-\frac{1}{\lambda} h(0+)\geq 0= F(\mu_0-).
		\end{equation*}
		Using \eqref{mip} and the fact that $h(+\infty)\leq 0$, we have
		\begin{equation}\label{gi}
			\lim_{x\rightarrow M-} F(x)=1-\frac{1}{\lambda}\cdot
			\lim_{x\rightarrow M-}h(\mu^{-1}(x))=1-\frac{1}{\lambda}\cdot
			h(+\infty)\leq 1.
		\end{equation}
		Therefore $F$ is monotone increasing on 
		$(-\infty,\infty)$. 
		Moreover, if $M<\infty$ then we have
		$F(+\infty)=1$ by \eqref{def_FF}, and if $M=+\infty$ then, by 
		\eqref{claim}, we have $h(+\infty)=0$, hence \eqref{gi} 
		becomes an equality, which again gives $F(+\infty)=1$. 
		We have therefore shown that $F(x)$ 
		is a cumulative probability distribution.
		
		It remains to show that $\mu(t)$ is 
		the optimal policy corresponding to the distribution $F$ we defined. Let us compute 
		the optimal policy $\tilde{\mu}(t)$  corresponding to $F,F_0$ (recall $F_0$ is an arbitrary distribution with expectation $\mu_0=\mu(0)$), and show that $\tilde{\mu}(t)=\mu(t)$.
		
		Using
		\eqref{def_FF}, and making the change of variable $u=\mu(t),\;\;du=\mu'(t)dt$, we have, for 
		$x\in [\mu_0,M)$,
		\begin{align*}\phi(x)&=\int_x^\infty (1-F(u))du
			=\frac{1}{\lambda}\int_x^M h(\mu^{-1}(u))du=\frac{1}{\lambda}\int_{\mu^{-1}(x)}^{\infty} h(t)\mu'(t)dt\\&=\frac{1}{\lambda}\int_{\mu^{-1}(x)}^{\infty} \mu''(t)dt=\frac{1}{\lambda}\mu'(\mu^{-1}(x)),
		\end{align*}
		$$\Psi(x)=\int_{\mu_0}^x \frac{du}{\phi(u)}=\lambda\int_{\mu_0}^x \frac{du}{\mu'(\mu^{-1}(u))}=\lambda\int_{\mu_0}^x [\mu^{-1}(u)]'du=\frac{1}{\lambda}\mu^{-1}(x)-\mu^{-1}(\mu_0)=\lambda\mu^{-1}(x).$$
		and since $\tilde{\mu}(t)=\Psi^{-1}(\lambda t)$ we have
		$$\lambda t=\Psi(\tilde{\mu}(t))=\lambda \mu^{-1}(\tilde{\mu}(t)),\;\;\Rightarrow \;\; \mu(t)=\tilde{\mu}(t),$$
		as we needed to show.

	\end{proof}

	\subsection{Inequalities for the optimal policy in terms of moments of the offer distribution}
	\label{moments}
	
	Since Theorem \ref{char} characterizes 
	optimal policy functions, we cannot hope to obtain any general properties of these functions beyond those listed in this theorem, or derived from them.
	However, in this subsection we will show that by suitably restricting the distribution $F(x)$, we can derive further  properties of the corresponding functions $\mu(t)$, in particular regarding their growth.
	
	The examples in subsection \ref{mu_examples}
	show offer distributions with different 
	types of tail behavior, whose corresponding 
	optimal policy functions $\mu(t)$ grow at different rates as the marketing period $t$ increases. We now derive some  bounds  
	relating tail behavior of the distribution $F(x)$ to the growth of the 
	corresponding optimal policy function $\mu(t)$.
	
	The results of this subsection have the advantage of providing bounds on $\mu(t)$ 
	valid for all $t$, but since they are only upper bounds, they do not 
	determine the precise asymptotic behavior 
	of $\mu(t)$. In subsection \ref{asymptotics1} we will derive several results which 
	require stronger assumptions on the 
	behavior of the offer distributions, and provide information on the 
	asymptotic behavior of $\mu(t)$.

	\begin{thm}
		Assume the offer distribution satisfies $E[X_+^p]<\infty$ ($p>1$). Then for any $\mu_0\geq 0$, the corresponding optimal policy satisfies, for all $t\geq 0$,
		\begin{equation}\label{ineq1}\mu(t)\leq \mu_0+E[(X-\mu_0)_+^p]^{\frac{1}{p}}\cdot (\lambda t)^{\frac{1}{p}}.\end{equation}
	\end{thm}
	
	In particular, this implies that 
	$\mu(t)=O(t^{\frac{1}{p}})$ as $t\rightarrow \infty$. Note that for the Pareto distribution \eqref{pareto} we have $E(X^p)<\infty$
	for all $p<\alpha$, and the exact result \eqref{mu_pareto_g} shows that $\mu(t)=O(t^{\frac{1}{\alpha}})$.
	
	\begin{proof} Assuming $x>\mu_0$, we have
		$$X>x \;\;\Leftrightarrow  \;\;(X-\mu_0)_+^p>(x-\mu_0)^p,$$	
		so, using Markov's inequality,
		$$1-F(x)=P(X>x)=P((X-\mu_0)_+^p>(x-\mu_0)^p)\leq \frac{E[(X-\mu_0)_+^p]}{(x-\mu_0)^p},$$
		hence, for $x>\mu_0$
		$$\phi(x)=\int_x^\infty (1-F(u))du\leq E[(X-\mu_0)_+^p]\int_{x}^\infty \frac{1}{(u-\mu_0)^p}du= \frac{E[(X-\mu_0)_+^p]}{(p-1)(x-\mu_0)^{p-1}},$$
		so that by \eqref{def_Psi}
		$$\Psi(x)=\int_{\mu_0}^{x}\frac{dv}{\phi(v)}\geq \frac{p-1}{E[(X-\mu_0)_+^p]}\int_{\mu_0}^x (v-\mu_0)^{p-1} dv= \frac{(x-\mu_0)^p}{E[(X-\mu_0)_+^p]}.$$
		Substituting $x=\mu(t)$ and recalling \eqref{ef},
		we have
		$$\lambda t=\Psi(\mu(t))\geq \frac{(\mu(t)-\mu_0)^p}{E[(X-\mu_0)_+^p]},$$
		which, solved for $\mu(t)$, gives \eqref{ineq1}.
	\end{proof}

	For distributions whose right tail decays at least exponentially, the following inequality shows that the growth of 
	$\mu(t)$ is at most logarithmic.
	
	\begin{thm} Assume the offer distribution satisfies $E[e^{\delta X_+}]<\infty$ ($\delta>0$). Then, for any $\mu_0\geq 0$, the corresponding optimal policy satisfies, for all $t\geq 0$,
		\begin{equation}\label{ineq2}\mu(t)\leq \mu_0+ \frac{1}{\delta}\ln\left(E[e^{\delta (X-\mu_0)_+}]\cdot \lambda t+1 \right).\end{equation}
	\end{thm}
	

\begin{proof}
	For $x>\mu_0$, we have, using Markov's inequality,
	$$1-F(x)=P((X-\mu_0)_+>x-\mu_0)=P(e^{\delta (X-\mu_0)_+}>e^{\delta (x-\mu_0)})\leq \frac{E[e^{\delta (X-\mu_0)_+}]}{e^{\delta (x-\mu_0)}}$$
	%
	hence, for $x>\mu_0$
	$$\phi(x)=\int_x^\infty (1-F(u))du\leq E[e^{\delta (X-\mu_0)_+}]\int_{x}^\infty e^{-\delta (u-\mu_0)}du= E[e^{\delta (X-\mu_0)_+}]\frac{1}{\delta}e^{-\delta (x-\mu_0)}$$
	$$\Psi(x)=\int_{\mu_0}^{x}\frac{dv}{\phi(v)}\geq \frac{\delta}{E[e^{\delta (X-\mu_0)_+}]}\int_{\mu_0}^x e^{\delta (v-\mu_0)} dv= \frac{e^{\delta (x-\mu_0)}-1}{E[e^{\delta (X-\mu_0)_+}]}\;\;\Rightarrow\;\; \lambda t=\Psi(\mu(t))\geq \frac{e^{\delta (\mu(t)-\mu_0)}-1}{E[e^{\delta (X-\mu_0)_+}]},$$
	which, solved for $\mu(t)$, gives \eqref{ineq2}.
\end{proof}

\section{Distribution of the realized sale price}
\label{price_distribution}

Recall that $S_t$ is the random variable 
describing the realized sale price, assuming the marketing period is $t$ and the
optimal policy is followed. 
Define the cumulative distribution function of $S_t$ by
$$G_t(x)=P(S_t\leq x).$$
Our aim is to obtain explicit expressions for $G_t(x)$ and for the corresponding probability density. In all the developments below, it is assumed that the
function $\mu(t)=E[S_t]$ is known, since we have already shown how to compute it in section \ref{optimal_policy0}. In subsection \ref{var_V}, we will use the expression for $G_t(x)$ to derive a formula for the variance $Var[S_t]$, and in subsection \ref{examples_distribution}
we will apply the results to specific offer distributions.

\subsection{Computing the sale price distribution}

We use a `first-step analysis', assuming we start when the remaining time is $t$ and the first offer $X$ arrives 
after time $D$. 

Assume first that $D\leq t$. If $X\geq \mu(t-D)$, the first offer is accepted and we thus have
\begin{eqnarray*}P(S_t\leq x\;|\;X\geq \mu(t-D))&=&P(X\leq x\;|\;D\leq t, X>\mu(t-D))=
	\begin{cases}
		0 & x\leq \mu(t-D)\\
		\frac{F(x)-F(\mu(t-D))}{1-F(\mu(t-D))} & x> \mu(t-D)
	\end{cases}.
\end{eqnarray*}
If $X<\mu(t-D)$, the first offer is rejected, and we wait for the next offer, so that 
$$P(S_t\leq x\;|\;X\leq \mu(t-D))=P(S_{t-D}\leq x)=G_{t-D}(x).$$
Thus
\begin{align*}P(S_t\leq x\;|\;D\leq t)&=P(X\geq \mu(t-D))\cdot P(S_t\leq x\;|\;X\geq \mu(t-D))+P(X< \mu(t-D))\cdot G_{t-D}(x)\\
	&=(1-F(\mu(t-D)))P(S_t\leq x\;|\;X\geq \mu(t-D))+F(\mu(t-D))G_{t-D}(x)\\
	&=\begin{cases}
		F(\mu(t-D))G_{t-D}(x) & x<\mu(t-D)\\
		F(x)-F(\mu(t-D))+F(\mu(t))G_{t-D}(x)& x\geq \mu(t-D)
	\end{cases}\\
	&= F(\mu(t-D))G_{t-D}(x)+ [F(x)-F(\mu(t-D))]_+.
\end{align*}
If $D>t$, then $S_t=X_0$, so
$$P(S_t\leq x\;|\;D>t)=P(X_0\leq x)=F_0(x).$$
Therefore
\begin{align}\label{Gd1}G_t(x)&=P(S_t\leq x)=\lambda \int_0^t e^{-\lambda s}\Big( F(\mu(t-s))G_{t-s}(x)+ [F(x)-F(\mu(t-s))]_+ \Big)ds+e^{-\lambda t}F_0(x)\nonumber\\
	&=\lambda e^{-\lambda t}\int_0^t e^{\lambda s}\Big( F(\mu(s))G_s(x)+ [F(x)-F(\mu(s))]_+ \Big)ds+e^{-\lambda t}F_0(x).\end{align}
Note that substituting $t=0$ into \eqref{Gd1} gives
$$G_0(x)=F_0(x).$$
Multiplying both sides of \eqref{Gd1} by
$e^{\lambda t}$ and then 
differentiating with respect to $t$, we obtain
$$e^{\lambda t}G_t(t,x)+\lambda e^{\lambda t}G_t(x)=\lambda e^{\lambda t}\Big( F(\mu(t))G_t(x)+ [F(x)-F(\mu(t))]_+ \Big),$$
which, after simplification, gives
\begin{lemma} For each fixed $x$, $G_t(x)$, as a function of $t$, satisfies the  initial value problem
	\begin{equation}\label{deg} \begin{cases}\frac{d}{dt}G_t(x)=\lambda  \left[(F(\mu(t))-1)G_t(x)+[F(x)-F(\mu(t))]_+ \right],\\
			G_0(x)=F_0(x) 
		\end{cases}.
	\end{equation}
\end{lemma}
In the following we will solve the above differential 
equation, which is a first-order linear differential equation with respect to $t$, with $x$ considered a parameter. We thus obtain an explicit expression for $G_t(x)$.
Note that the expression
for the distribution of $S_t$ takes
different functional forms in the intervals 
$(-\infty,\mu_0]$, $[\mu_0,\mu(t)]$, and 
$[\mu(t),\infty)$. 

\begin{thm}
	The cumulative density of the sale price $S_t$ is given by
	\begin{align}\label{Gdist}G_t(x)&=
		\phi(\mu(t))\left[ \frac{F_0(x)}{\phi(\mu_0)}+\int_{\mu_0}^{\mu(t)} \frac{[F(x)-F(w)]_+}{\phi(w)^2} dw\right]\\
		&=\phi(\mu(t)) \frac{F_0(x)}{\phi(\mu_0)}+\phi(\mu(t))\cdot 
		\begin{cases}
			0 &  x<\mu_0\\
			\frac{1}{\phi(x)}-\frac{1}{\phi(\mu_0)}+\phi'(x)\int_{\mu_0}^{x}\frac{dw}{\phi(w)^2}    & \mu_0\leq x<\mu(t)\\
			\frac{1}{\phi(\mu(t))}-\frac{1}{\phi(\mu_0)}+\phi'(x)\int_{\mu_0}^{\mu(t)}\frac{dw}{\phi(w)^2}	& x\geq \mu(t)
		\end{cases}.\nonumber
	\end{align}
	In case $F$,$F_0$ are absolutely continuous with densities $f,f_0$, the probability density function $g_t(x)$  of $S_t$ is given by
	\begin{equation}\label{density}g_t(x)
		=\phi(\mu(t)) \frac{f_0(x)}{\phi(\mu_0)}+\phi(\mu(t)) f(x)\cdot \begin{cases}
			0 & x<\mu_0\\
			\int_{\mu_0}^{x}\frac{dw}{\phi(w)^2}  & \mu_0\leq x<\mu(t)\\
			\int_{\mu_0}^{\mu(t)}\frac{dw}{\phi(w)^2}	& x>\mu(t)
		\end{cases}
		.\end{equation}
\end{thm}

\begin{proof}
	For each fixed $x$, \eqref{deg} in a linear differential equation of first order, whose general solution is therefore
	$$G_t(x)=C(x)e^{-\lambda \int_0^t[1-F(\mu(s))]ds}+\lambda \int_0^t e^{-\lambda \int_r^t[1-F(\mu(s))]ds} [F(x)-F(\mu(r))]_+dr,$$
	where $C(x)$ is the integration constant.
	The initial condition in \eqref{deg} implies 
	$C(x)=F_0(x)$. Defining 
	$$\Phi(t)=e^{-\lambda \int_{0}^t[1-F(\mu(s))]ds},$$
	we can write 
	\begin{equation}\label{G1}G_t(x)=F_0(x)\Phi(t)+\lambda \int_0^t \frac{\Phi(t)}{\Phi(r)} [F(x)-F(\mu(r))]_+dr.\end{equation}
	Using the change of variable
	$$w=\mu(s),\;\;dw=\mu'(s)ds=\lambda \phi(\mu(s))ds=\lambda \phi(w)ds$$
	we have
	\begin{align}\label{ii}
		\int_{0}^t[1-F(\mu(s))]ds=\frac{1}{\lambda}\int_{\mu_0}^{\mu(t)}[1-F(w)]\frac{dw}{\phi(w)}=-\frac{1}{\lambda}\int_{\mu_0}^{\mu(t)}\frac{\phi'(w)}{\phi(w)}dw=\frac{1}{\lambda}\ln\left( \frac{\phi(\mu_0)}{\phi(\mu(t))}\right),
	\end{align}
	hence 
	$$\Phi(t)=\frac{\phi(\mu(t))}{\phi(\mu_0)}.$$
	Therefore we can write \eqref{G1} in the form
	\begin{equation}\label{G20}G_t(x)=F_0(x)\frac{\phi(\mu(t))}{\phi(\mu_0)}+\lambda \int_0^t \frac{\phi(\mu(t))}{\phi(\mu(r))} [F(x)-F(\mu(r))]_+dr.\end{equation}
	Making a change of variable $w=\mu(r)$ in
	the integral above, we obtain
	$$\int_0^t \frac{[F(x)-F(\mu(r))]_+}{\phi(\mu(r))} dr=\frac{1}{\lambda}\int_{\mu_0}^{\mu(t)} \frac{[F(x)-F(w)]_+}{\phi(w)^2} dw,$$
	so that
	\begin{equation}\label{G2}G_t(x)=\phi(\mu(t))\left[ \frac{F_0(x)}{\phi(\mu_0)}+\int_{\mu_0}^{\mu(t)} \frac{[F(x)-F(w)]_+}{\phi(w)^2} dw\right],\end{equation}
	and we have the first expression in \eqref{Gdist}.
	To obtain the second expression, we 
	note first that by the monotonicity of $F$
	we have
	\begin{equation}\label{mono}
		[F(x)-F(w)]_+=\begin{cases}
			0 & x\leq w\\
			F(x)-F(w) & x>w
		\end{cases}.
	\end{equation}
	In the following we divide the calculation into three cases, $x<\mu_0$,$\mu_0\leq x<\mu(t)$, and $x\geq \mu(t)$.
	In the case $x<\mu_0$, \eqref{mono} implies that the integral in \eqref{G20} vanishes, so that 
	\begin{equation}\label{part1}G_t(x)=\phi(\mu(t))\frac{F_0(x)}{\phi(\mu_0)},\;\;\;x<\mu_0,
	\end{equation}
	and differentiating with respect to $x$ (assuming $F_0$ is absolutely continuous) we have
	\begin{equation}\label{part1d}g_t(x)=\frac{d}{dx}G_t(x)=\phi(\mu(t))\frac{f_0(x)}{\phi(\mu_0)},\;\;\;x<\mu_0.
	\end{equation}
	When
	$\mu_0\leq x<\mu(t)$, we have, using \eqref{mono}, 
	\begin{align}\label{kc}&\int_{\mu_0}^{\mu(t)} \frac{[F(x)-F(w)]_+}{\phi(w)^2}  dw=\int_{\mu_0}^{x} \frac{F(x)-F(w)}{\phi(w)^2} dw\\
		&=F(x)\int_{\mu_0}^{x}\frac{dw}{\phi(w)^2} -\int_{\mu_0}^{x} \frac{F(w)}{\phi(w)^2} dw=F(x)\int_{\mu_0}^{x}\frac{dw}{\phi(w)^2} -\int_{\mu_0}^{x} \frac{\phi'(w)+1}{\phi(w)^2} dw\nonumber\\
		&=(F(x)-1)\int_{\mu_0}^{x}\frac{dw}{\phi(w)^2} +\frac{1}{\phi(x)}-\frac{1}{\phi(\mu_0)}=\phi'(x)\int_{\mu_0}^{x}\frac{dw}{\phi(w)^2} +\frac{1}{\phi(x)}-\frac{1}{\phi(\mu_0)},\nonumber\end{align}
	so that \eqref{G2} reduces to
	\begin{equation}\label{part2}G_t(x)=\phi(\mu(t))\left[\frac{1}{\phi(x)}+\phi'(x)\int_{\mu_0}^{x}\frac{dw}{\phi(w)^2}-\frac{1-F_0(x)}{\phi(\mu_0)} \right],\;\;\;\mu_0\leq x<\mu(t)
	\end{equation}
	and differentiating with respect to $x$ (assuming $F,F_0$ are absolutely continuous) gives
	\begin{align}\label{part2d}g_t(x)=\frac{d}{dx}G_t(x)&=\phi(\mu(t))\left[-\frac{\phi'(x)}{\phi(x)^2}+\phi''(x)\int_{\mu_0}^{x}\frac{dw}{\phi(w)^2}
		+\phi'(x)\frac{1}{\phi(x)^2}
		+\frac{f_0(x)}{\phi(\mu_0)} \right]\\
		&=\phi(\mu(t))\left[f(x)\int_{\mu_0}^{x}\frac{dw}{\phi(w)^2}
		+\frac{f_0(x)}{\phi(\mu_0)} \right],\;\;\;\mu_0\leq x<\mu(t)\nonumber
	\end{align}
	where in the second line we have used the fact that $\phi''(x)=f(x)$.
	
	Finally, when $x\geq \mu(t)$, we have
	$$\int_{\mu_0}^{\mu(t)} \frac{[F(x)-F(w)]_+}{\phi(w)^2}  dw=\int_{\mu_0}^{\mu(t)} \frac{F(x)-F(w)}{\phi(w)^2} dw,$$
	and a calculation identical to t\eqref{kc} above leads to 
	$$\int_{\mu_0}^{\mu(t)} \frac{F(x)-F(w)}{\phi(w)^2} dw=\phi'(x)\int_{\mu_0}^{\mu(t)}\frac{dw}{\phi(w)^2} +\frac{1}{\phi(\mu(t))}-\frac{1}{\phi(\mu_0)},$$
	so that \eqref{G2} reduces to
	\begin{equation}\label{part3}G_t(x)=\phi(\mu(t))\left[\phi'(x)\int_{\mu_0}^{\mu(t)}\frac{dw}{\phi(w)^2}+\frac{1}{\phi(\mu(t))}-\frac{1-F_0(x)}{\phi(\mu_0)}\right],\;\;\; x\geq \mu(t)\end{equation}
	and differentiating with respect to $x$
	(assuming $F,F_0$ are absolutely continuous) gives 
	\begin{equation}\label{part3d}g_t(x)=\frac{d}{dx}G_t(x)=\phi(\mu(t))\left[f(x)\int_{\mu_0}^{\mu(t)}\frac{dw}{\phi(w)^2}+\frac{f_0(x)}{\phi(\mu_0)}\right],\;\;\;x\geq\mu(t).\end{equation}
	Combining \eqref{part1},\eqref{part2} and \eqref{part3} gives \eqref{Gdist} and combining 
	\eqref{part1d},\eqref{part2d} and \eqref{part3d}
	gives \eqref{density}. 
\end{proof}

\subsection{Variance of the sale price}
\label{var_V}

Since we now have 
an explicit expression for the distribution of the sale price $S_t$, we can derive its variance, which allows us to evaluate the uncertainty in the realized sale price when the optimal policy is applied.

\begin{thm}
	Assume $E(X_0^2)<\infty$ and $E(X_+^2)<\infty$.
	Then the variance of the sale price, under the optimal policy, is finite and given by
	\begin{equation}\label{varf}
		Var[S_t]=\phi(\mu(t)) \left[\frac{Var[X_0]}{\phi(\mu_0)} +2
		\int_{\mu_0}^{\mu(t)}\frac{1}{\phi(w)^2} \int_{w}^\infty \phi(x) dx dw\right], \end{equation}
	which can also be written as
	\begin{equation}\label{varfv}
		Var[S_t]=E[(X-\mu(t))_+] \left[\mu(t)-\mu_0 +\frac{Var[X_0]}{E((X-\mu_0)_+)}+
		\int_{\mu_0}^{\mu(t)}\frac{Var[(X-w)_+]}{E[(X-w)_+]^2}  dw\right].
	\end{equation}
\end{thm}

\begin{proof}
	We will be using the following identity:
	if $Z$ is a random variable with CDF $F_Z$, then 
	\begin{equation}\label{identity}E[Z^2]=2\int_0^\infty z [1-F_Z(z)]dz+2\int_{-\infty}^0 zF_Z(z) dz.\end{equation}
	Using \eqref{defphi} and \eqref{identity}, with $Z=(X-w)_+$ ($w\geq 0$) we have $F_Z(z)=F(z+w)$ for $z\geq 0$, hence
	\begin{align}\label{pw}&\int_w^\infty \phi(x)dx=\int_w^\infty \int_x^\infty (1-F(u))du dx =\int_w^\infty \int_w^u (1-F(u))dx du
		=\int_w^\infty (u-w)(1-F(u)) du \\
		&=\int_0^\infty u(1-F(z+w)) dz=\int_0^\infty z(1-F_Z(z))dz=\frac{1}{2}E(Z^2) =\frac{1}{2}E[(X-w)_+^2], \nonumber
	\end{align}
	and since $E[(X-w)_+^2)]\leq E[X_+^2]<\infty$, by our assumption, this shows in particular that the integral on the left-hand side of \eqref{pw} is finite.
	
	Using \eqref{identity} with $Z=S_t$ and 
	\eqref{Gdist}, we have
	\begin{align}\label{comp1}&E[S_t^2]
		=2\int_0^\infty x[1-G_t(x)]dx+2\int_{-\infty}^0 xG_t(x)dx\\ 
		&=2\int_{0}^{\mu_0}\left[1-\phi(\mu(t))\frac{F_0(x)}{\phi(\mu_0)} \right]xdx+2\int_{\mu_0}^{\mu(t)}\left[1-\phi(\mu(t))\left(\frac{F_0(x)}{\phi(\mu_0)}+ \frac{1}{\phi(x)}-\frac{1}{\phi(\mu_0)}+\phi'(x)\int_{\mu_0}^{x}\frac{dw}{\phi(w)^2}\right) \right]xdx\nonumber\\
		&+2\int_{\mu(t)}^\infty \left[1-\phi(\mu(t))\left(\frac{F_0(x)}{\phi(\mu_0)}+	\frac{1}{\phi(\mu(t))}-\frac{1}{\phi(\mu_0)}+\phi'(x)\int_{\mu_0}^{\mu(t)}\frac{dw}{\phi(w)^2}\right) \right]xdx+2\int_{-\infty}^{0}\phi(\mu(t))\frac{F_0(x)}{\phi(\mu_0)}xdx\nonumber\\
		&=2\int_{0}^{\mu_0}xdx-2\phi(\mu(t))\left[\int_{\mu_0}^{\mu(t)}\left[\frac{1}{\phi(x)}+ \phi'(x)\int_{\mu_0}^{x}\frac{dw}{\phi(w)^2}\right]xdx+\int_{\mu(t)}^\infty 	\phi'(x)\int_{\mu_0}^{\mu(t)}\frac{dw}{\phi(w)^2} xdx\right]\nonumber\\
		&+2\int_{\mu_0}^{\mu(t)}xdx+\frac{2\phi(\mu(t))}{\phi(\mu_0)}\left[\int_{0}^\infty x [1-F_0(x)]dx+\int_{-\infty}^0 xF_0(x)dx-\int_0^{\mu_0}xdx\right]\nonumber\\
		&=-2\phi(\mu(t))\left[\int_{\mu_0}^{\mu(t)}\left[\frac{1}{\phi(x)}+ \phi'(x)\int_{\mu_0}^{x}\frac{dw}{\phi(w)^2}\right]xdx+\int_{\mu(t)}^\infty 	\phi'(x)\int_{\mu_0}^{\mu(t)}\frac{dw}{\phi(w)^2} xdx\right]\nonumber\\
		&+\mu(t)^2+\frac{\phi(\mu(t))}{\phi(\mu_0)}\left(E[X_0^2]-E[X_0]^2\right).\nonumber
	\end{align}
	We have
	\begin{align}\label{cc1}
		&\int_{\mu_0}^{\mu(t)}\left[\frac{1}{\phi(x)}+ \phi'(x)\int_{\mu_0}^{x}\frac{dw}{\phi(w)^2}\right]xdx=\int_{\mu_0}^{\mu(t)} x\frac{d}{dx}\left[\phi(x)\int_{\mu_0}^{x}\frac{dw}{\phi(w)^2}\right]dx\\
		&= x\left[\phi(x)\int_{\mu_0}^{x}\frac{dw}{\phi(w)^2}\right]\Big|_{x=\mu_0}^{x=\mu(t)}-\int_{\mu_0}^{\mu(t)}\left[\phi(x)\int_{\mu_0}^{x}\frac{dw}{\phi(w)^2}\right]dx\nonumber\\
		&=\mu(t)\phi(\mu(t))\int_{\mu_0}^{\mu(t)}\frac{dw}{\phi(w)^2} -\int_{\mu_0}^{\mu(t)}\left[\phi(x)\int_{\mu_0}^{x}\frac{dw}{\phi(w)^2}\right]dx,\nonumber
	\end{align}
	and
	\begin{align}\label{cc2}
		&\int_{\mu(t)}^\infty 	\phi'(x)\int_{\mu_0}^{\mu(t)}\frac{dw}{\phi(w)^2} xdx=\left[\phi(x)x\Big|_{x=\mu(t)}^{x=\infty}-\int_{\mu(t)}^\infty \phi(x)dx\right]
		\int_{\mu_0}^{\mu(t)}\frac{dw}{\phi(w)^2}\nonumber\\&=-\left[\phi(\mu(t))\mu(t)+\int_{\mu(t)}^\infty \phi(x)dx\right]
		\int_{\mu_0}^{\mu(t)}\frac{dw}{\phi(w)^2}.
	\end{align}
	Combining \eqref{cc1} and \eqref{cc2} gives
	\begin{align}\label{comp3}
		&\int_{\mu_0}^{\mu(t)}\left[\frac{1}{\phi(x)}+ \phi'(x)\int_{\mu_0}^{x}\frac{dw}{\phi(w)^2}\right]xdx+\int_{\mu(t)}^\infty 	\phi'(x)\int_{\mu_0}^{\mu(t)}\frac{dw}{\phi(w)^2} xdx \\
		&=-\left[   \int_{\mu_0}^{\mu(t)}\left[\phi(x)\int_{\mu_0}^{x}\frac{dw}{\phi(w)^2}\right]dx+
		\int_{\mu(t)}^\infty\phi(x)\int_{\mu_0}^{\mu(t)}\frac{1}{\phi(w)^2} dwdx \right]\nonumber\\
		&=-\left[  \int_{\mu_0}^{\mu(t)}\int_{w}^{\mu(t)}\phi(x)\frac{1}{\phi(w)^2}dx dw+ 
		\int_{\mu_0}^{\mu(t)}\int_{\mu(t)}^\infty\phi(x)\frac{1}{\phi(w)^2} dx dw \right] =-
		\int_{\mu_0}^{\mu(t)}\frac{1}{\phi(w)^2} \int_{w}^\infty \phi(x) dx dw. \nonumber
	\end{align}
	From \eqref{comp1},\eqref{comp3} we get
	$$E[S_t^2]=\mu(t)^2+\frac{\phi(\mu(t))}{\phi(\mu_0)}\cdot Var[X_0]+2\phi(\mu(t)) 
	\int_{\mu_0}^{\mu(t)}\frac{1}{\phi(w)^2} \int_{w}^\infty \phi(x) dx dw,$$
	which, since $\mu(t)^2=E[S_t]^2$, implies \eqref{varf}.
	
	To obtain \eqref{varfv}, note that, using \eqref{pw} and \eqref{prob_phi} we have
	\begin{align}\label{idk}
		&2
		\int_{\mu_0}^{\mu(t)}\frac{1}{\phi(w)^2} \int_{w}^\infty \phi(x) dx dw=\int_{\mu_0}^{\mu(t)}\frac{E[(X-w)_+^2]}{E[(X-w)_+]^2}  dw\\&=(\mu(t)-\mu_0)+\int_{\mu_0}^{\mu(t)}\frac{E[(X-w)_+^2]-E[(X-w)_+]^2}{E[(X-w)_+]^2}  dw
		=(\mu(t)-\mu_0)+\int_{\mu_0}^{\mu(t)}\frac{Var[(X-w)_+]}{E[(X-w)_+]^2}  dw,\nonumber
	\end{align}
	which, together with \eqref{varf} and \eqref{prob_phi}, gives \eqref{varfv}.
	
\end{proof}

\subsection{Examples}
\label{examples_distribution}

In the following examples we 
compute expressions for the sale price distribution 
corresponding to three offer distributions, using the formulas derived above.

In contrast with the results for the 
expectation $\mu(t)=E[S_t]$ in subsection 
\ref{mu_examples}, in which the residual
distribution $F_0(x)$ was arbitrary (with a given expectation $\mu_0$), the full  distribution of $S_t$ depends 
on both on $F(x)$ and on $F_0(x)$.
In our examples below, we will assume that
that $F_0=F$, which means that if the deadline is reached before a sale has been made, one waits for the next offer and accepts it unconditionally. In particular this implies that
$$\mu_0=E[X_0]=E[X].$$
One can easily modify the calculations for $F_0\neq F$, as desired: for example, in the case when the sale is lost if the deadline is reached, the corresponding results will be obtained by setting $F_0(x)=1$, $\mu_0=0$.

We also perform simulations of the bidding process, using
the optimal policy, and demonstrate 
the perfect fit of the analytical results obtained above with the simulation results. By generating many such
simulations we obtain a histogram of the realized sale prices 
attained, and can compare these histograms to the 
analytical expressions for the probability density 
functions of the sale price. Figure \ref{fcode} displays simple Matlab code for
generating simulations of the bidding process.

%
%
%
%
%
%
%

\begin{figure}[h!]
	\begin{lstlisting}
		function [S,T]=simulate(lambda,t,a,b,N) % a,b - parameters for uniform distribution, N - number of simulations
		
    	  	for n=1:N 
	        	t_rem=t; % initiate remaining time
	        	accept=0; % flag for offer being accepted
	        	while t_rem>0 && accept==0
	             	r=exprnd(1/lambda); % time to arrival of offer
	            	t_rem=t_rem-r; % update remaining time
	            	if t_rem>0
	                 	th=b-2*(b-a)/(lambda*t_rem+4); % compute threshold price
	                 	X=unifrnd(a,b); % draw offer
	                 	if X>=th % accept offer if above threshold
	                         	accept=1;
	                        	S(n)=X;
	                        	T(n)=t-t_rem;
	                 	end  
	            	end
	        	end
	        	if accept==0  %  deadline reached before accepting offer
	            	S(n)=unifrnd(a,b); % draw from residual distribution
	            	T(n)=t;
    	    	end
	    	end
		end
		\end{lstlisting}
	\caption{Matlab code for simulating a bidding process $N$ times, using the optimal policy, when the offer distribution and the residual distribution are uniform on $[a,b]$. The function outputs vectors $S,T$ of length $N$ (number of simulations), with $S$ containing the realized sale prices and $T$ containing the time durations until the sale was made.}
	\label{fcode}
\end{figure}

\subsubsection{Uniform offer distribution}
\label{dist_uniform}

We continue the example of subsection
\ref{policy_uniform}.
Using \eqref{phi_uniform} and 
\eqref{mu_uniform_g} with $\mu_0=E(X_0)=E(X)=\frac{a+b}{2}$ 
we have
\begin{equation*}\label{mu_uniform}
	\mu(t)=b-\frac{2(b-a)}{\lambda t +4},
\end{equation*}
\begin{equation}\label{uni_phi_mu}\phi(\mu(t))=\frac{2(b-a)}{(\lambda t +4)^2},\;\;\;\phi(\mu_0)=\frac{b-a}{8},\end{equation}
\begin{equation}\label{int0}\int_{\mu_0}^{x}\frac{dw}{\phi(w)^2}=\frac{4}{3}(b-a)^2\left[\frac{1}{(b-x)^3}-\frac{8}{(b-a)^3}  \right],\;\;\;\int_{\mu_0}^{\mu(t)}\frac{dw}{\phi(w)^2}=\frac{1}{6}\frac{1}{b-a}\left[\left(\lambda t +4\right)^3-64  \right].
\end{equation}
Plugging these expressions 
into \eqref{density} we obtain
\begin{equation}\label{uni_den}g_t(x)
	= \frac{1}{(b-a)\cdot (\lambda t +4)^2}\cdot  \begin{cases}
		0 & x<a\\
		16 & a\leq x<\frac{a+b}{2}\\
		\frac{8}{3}\left[\frac{(b-a)^3}{(b-x)^3}-2  \right] & \frac{a+b}{2}\leq x<b-\frac{2(b-a)}{\lambda t +4}\\
		\frac{1}{3}\left[\left(\lambda t +4\right)^3-16  \right]	& b-\frac{2(b-a)}{\lambda t +4}\leq x<b\\
		0 & x\geq b
	\end{cases}
	.\end{equation}
Examples of plots of this density function, for the case $a=1,b=3$, 
together with a histogram of simulation results for the same parameters, is given in Figure \ref{f2}.
Note that the support of the distribution is 
divided into $3$ intervals, whose lengths depend on $t$, with the density constant on the left and right interval, and given by a rational function on the middle interval.
\begin{figure}
	\centering
	\includegraphics[height=6cm, angle=0]{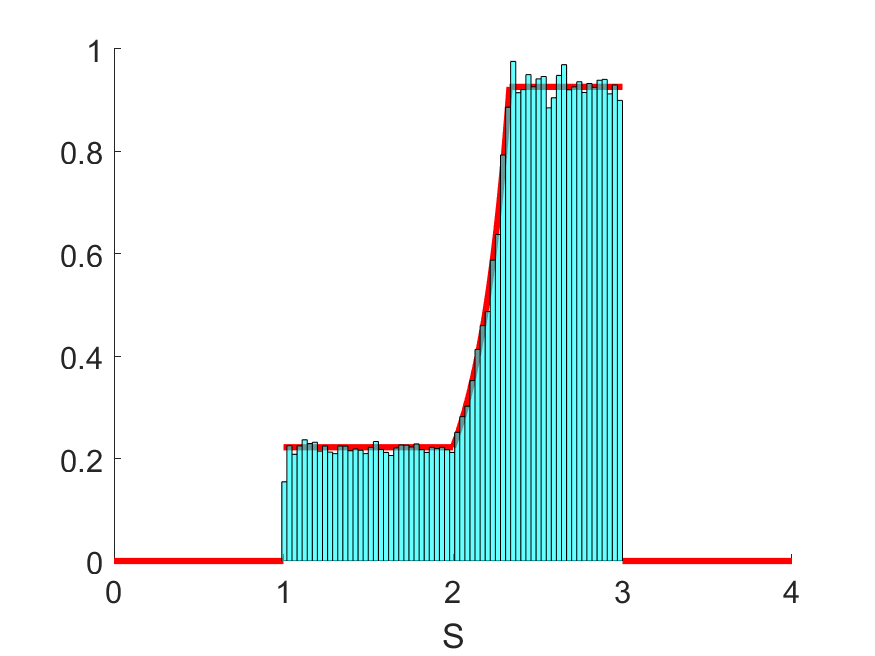}
	\includegraphics[height=6cm, angle=0]{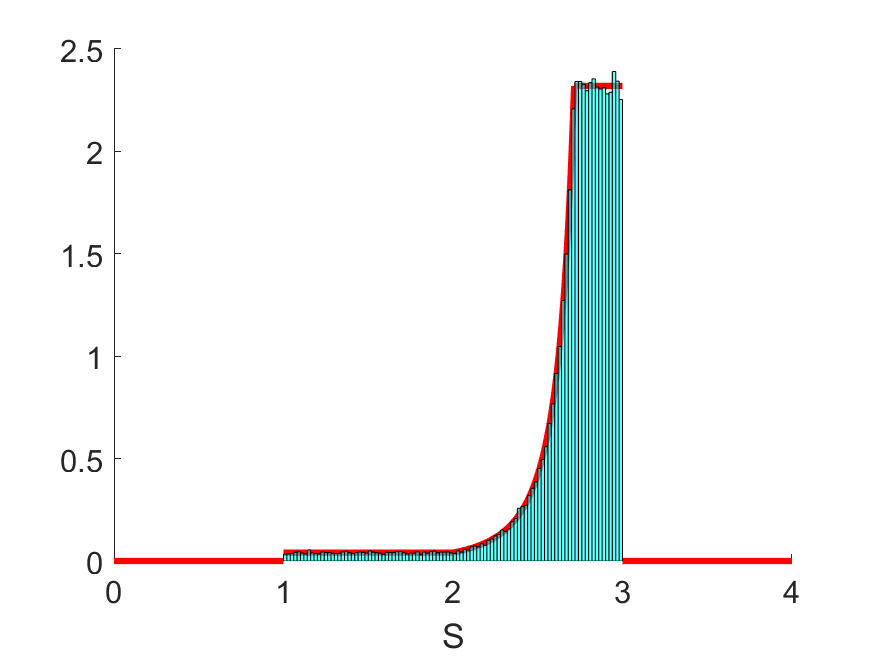}
	\caption{Histograms of sale prices attained in $N=10^5$ simulations of the bidding process, employing the optimal policy, where $\lambda=1$, and the offer distribution and residual distributions are uniform on $[1,3]$. The red line shows the analytical expression for the probability density, given by \eqref{uni_den}. Left: t=2, Right: t=10.}
	\label{f2}
\end{figure}


Using \eqref{varf} to compute the variance of $S_t$, we obtain
\begin{equation}\label{uni_var}Var[S_t]=\frac{4(b-a)^2}{3(\lambda t +4)^2} \left[2
	\ln\left(\frac{1}{4}\lambda t+1 \right)+1\right].\end{equation}
It can be checked that $Var[S_t]$ 
is monotone decreasing in $t$, with $Var[S_t]\rightarrow 0$ 
as $t\rightarrow \infty$.

\subsubsection{Exponential offer distribution}
\label{dist_exponential}

We continue the example of subsection \ref{policy_exponential}, assuming now 
$F_0(x)=F(x)$, so that $\mu_0=\eta$.
Using \eqref{phi_exp} and \eqref{mu_exp}, we have
\begin{equation}\label{exp_phi_mu}\phi(\mu(t))=\frac{\eta}{\lambda t+e},\end{equation}
$$\int_{\mu_0}^x\frac{dw}{\phi(w)^2}=\frac{1}{\eta^2}\int_{\eta}^xe^{\frac{2w}{\eta}}dw=\frac{1}{2\eta}\left[e^{\frac{2x}{\eta}}-e^{2} \right],\;\;\;\int_{\mu_0}^{\mu(t)} \frac{dw}{\phi(w)^2}=\frac{1}{2\eta}\lambda t\left[\lambda t+2e \right],$$
$$\phi(\mu_0)=\eta e^{-1},\;\;\;1-F_0(x)=e^{-\frac{x}{\eta}},$$
and substituting these expressions into
\eqref{density} gives
	
	\begin{equation}\label{exp_den}g_t(x)=\begin{cases}
			0 & x<0\\
			\frac{1}{\eta}\cdot \frac{e}{\lambda t+e}\cdot e^{-\frac{x}{\eta}} & 0\leq x<\eta\\
			\frac{1}{\eta}\cdot \frac{e}{\lambda t+e}\left[\frac{1}{2}e^{\frac{x}{\eta}-1}-\frac{1}{2}e^{1-\frac{x}{\eta}}+e^{-\frac{x}{\eta}} \right]   & \eta\leq x<\eta \ln\left(\lambda t+e \right)\\
			\frac{1}{2\eta}\left[ \lambda t+e -\frac{(e-2)e}{\lambda t+e}\right]e^{-\frac{x}{\eta}}	& x\geq\eta \ln\left(\lambda t+e \right).
	\end{cases}\end{equation}
	Figure \ref{f3} shows examples of this 
	density function for two values of $t$, together with simulation results.
	
	\begin{figure}
		\centering
		\includegraphics[height=6cm, angle=0]{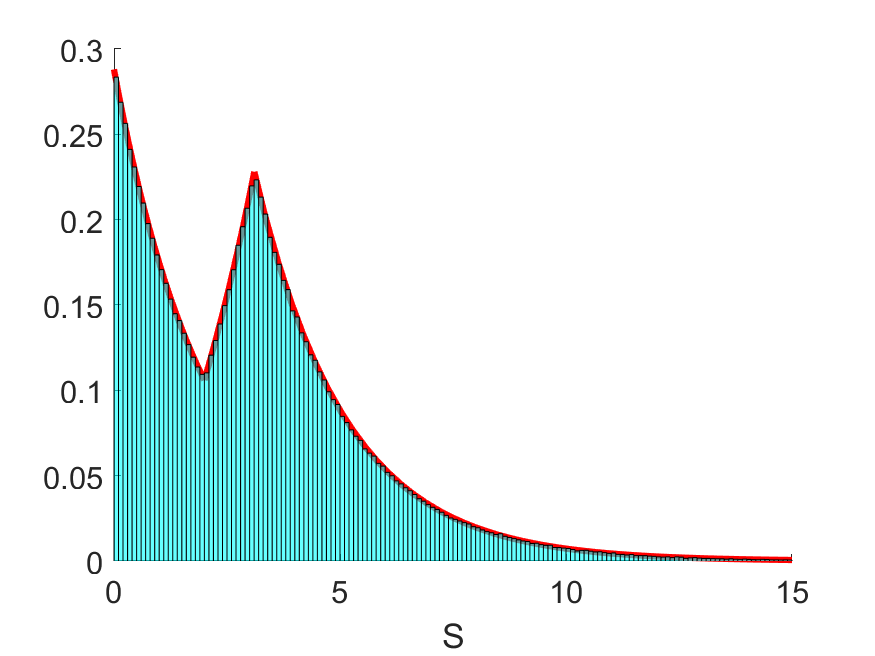}
		\includegraphics[height=6cm, angle=0]{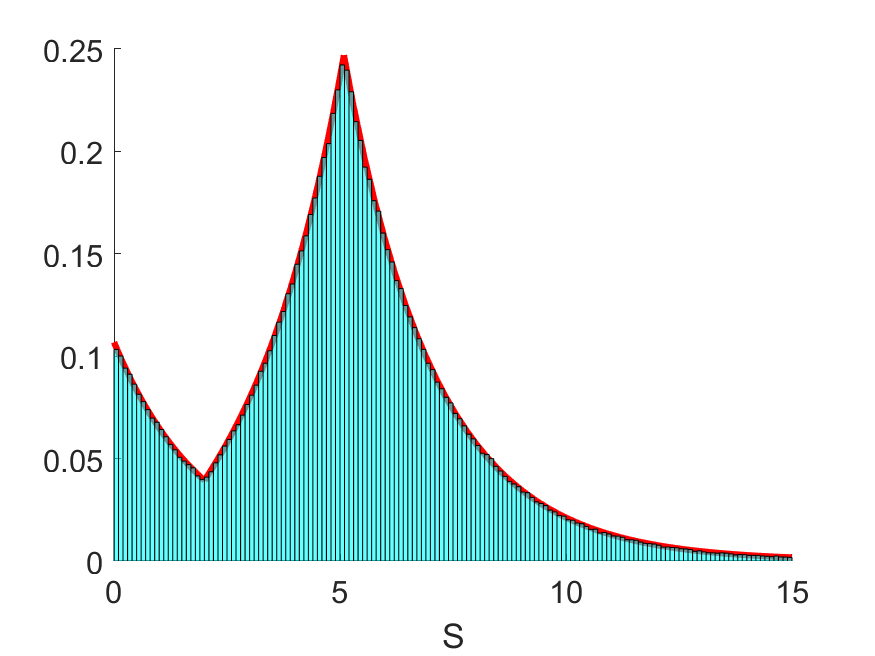}
		\caption{Histograms of sale prices attained in $N=10^5$ simulations of the bidding process, employing the optimal policy, where $\lambda=1$ and the offer and residual distributions are exponential with mean $\eta=2$. The red line shows the analytical expression for the probability density, given by \eqref{exp_den}.
		Left: $t=2$, Right: $t=10$.}
		\label{f3}
	\end{figure}
	
	Using \eqref{varf} to compute the
	variance of $S_t$, we obtain:
	\begin{equation}\label{var_expo}	Var[S_t]=\left[1+\frac{
			\lambda t}{\lambda t+e}\right]\cdot \eta^2.\end{equation}
	Note that as $t$ varies from $0$ to 
	$+\infty$, $Var[S_t]$ increases from
	$\eta^2$ to $2\eta^2$.

	\subsubsection{Pareto offer distribution}
	
	\begin{figure}
		\centering
		\includegraphics[height=6cm, angle=0]{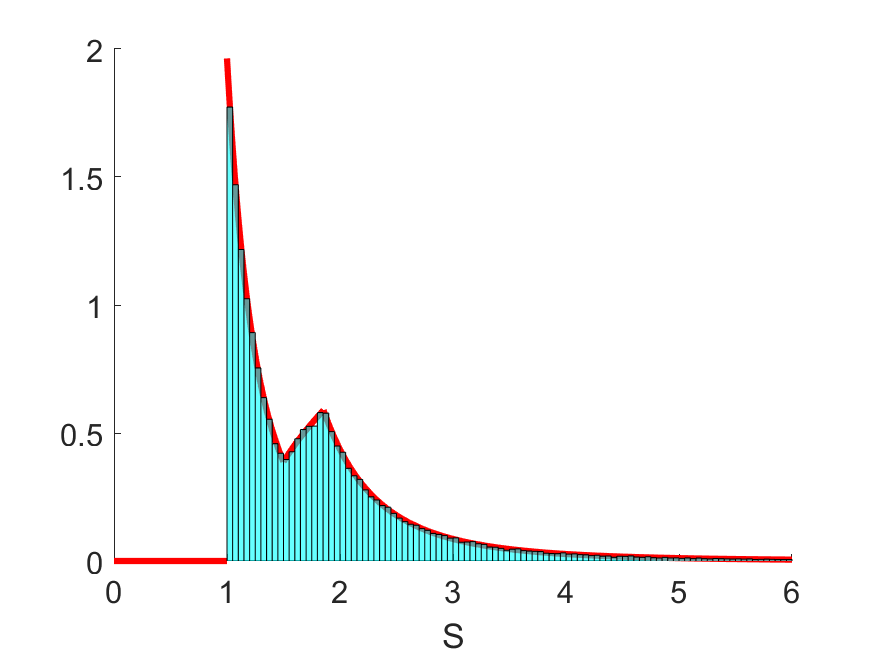}
		\includegraphics[height=6cm, angle=0]{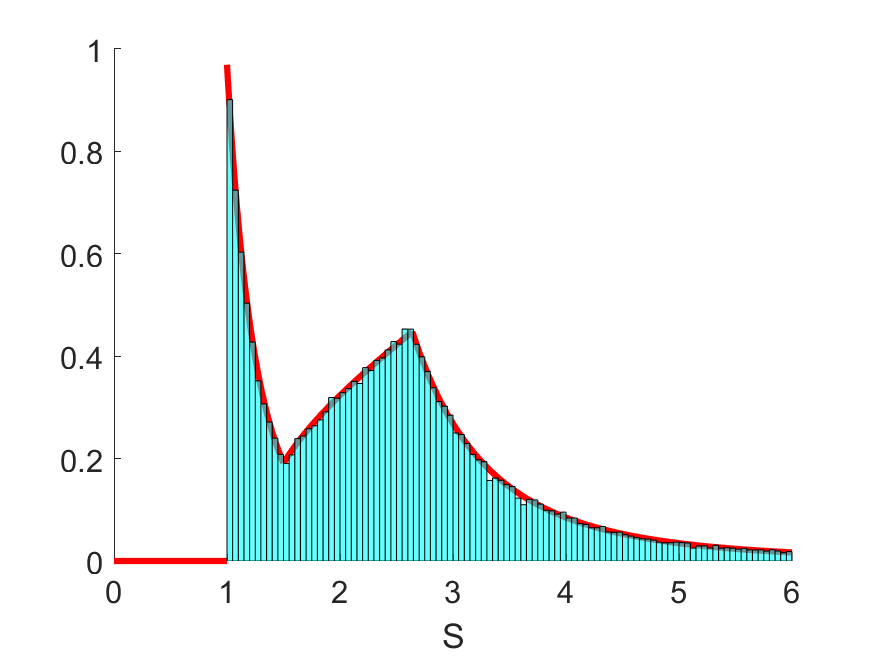}
		\caption{Histograms of sale prices attained in $N=10^5$ simulations of the bidding process, employing the optimal policy, where $\lambda=1$ and the offer and residual distributions are Pareto with $x_m=1,\alpha=3$. The red line shows the analytical expression for the probability density, given by \eqref{pareto_den}. Left: $t=2$, Right: $t=10$.}
		\label{f4}
	\end{figure}
	
	We continue the example of subsection \ref{policy_pareto}, now setting
	$$F_0(x)=F(x),\;\;\mu_0=E(X_0)=\frac{\alpha}{\alpha-1}\cdot x_m.$$
	By \eqref{phi_pareto},\eqref{mu_pareto_g},\eqref{defc} we have
	$$c=\frac{\alpha}{\alpha-1}\left(\frac{x_m}{\mu_0} \right)^\alpha=\left(\frac{\alpha-1}{\alpha} \right)^{\alpha-1},$$
	$$
	\mu(t)=\mu_0\left[c\lambda t+1\right]^{\frac{1}{\alpha}},$$
	\begin{equation}\label{cp}\phi(x)=\frac{x_m^\alpha}{\alpha-1}\cdot \frac{1}{x^{\alpha-1}},\;\;\;\phi(\mu(t))=\frac{1}{\lambda} \mu'(t)=\frac{1}{\alpha}\mu_0c\left[c \lambda t+1\right]^{\frac{1}{\alpha}-1},\;\;\;\phi(\mu_0)=\frac{\mu_0 c}{\alpha},\end{equation}
	$$\int_{\mu_0}^x \frac{dw}{\phi(w)^2}=\frac{(\alpha-1)^2}{x_m^{2\alpha}}\int_{\mu_0}^x w^{2\alpha-2}=\frac{(\alpha-1)^2}{2\alpha-1}\frac{1}{x_m^{2\alpha}}\left[x^{2\alpha-1}-\mu_0^{2\alpha-1} \right]=\frac{\alpha^2}{2\alpha-1}\frac{1}{\mu_0 c^2}\left[\left(\frac{x}{\mu_0}\right)^{2\alpha-1}-1 \right],$$
	$$\int_{\mu_0}^{\mu(t)} \frac{dw}{\phi(w)^2}=\frac{\alpha^2}{2\alpha-1}\frac{1}{\mu_0 c^2}\left[[c\lambda t+1]^{\frac{2\alpha-1}{\alpha}}-1 \right].$$
	Plugging the above expressions into 
	\eqref{density} gives us the probability density for the 
	sale price ($x\geq x_m$)
	\begin{equation}
		\label{pareto_den}
	g_t(x)
	=\frac{(\alpha-1) \mu_0^\alpha}{\left[c \lambda t+1\right]^{1-\frac{1}{\alpha}}}\cdot \frac{1}{x^{\alpha+1}} \left[c+\frac{\alpha}{2\alpha-1} \cdot \begin{cases}
		0 & x<\mu_0\\
		\left(\frac{x}{\mu_0}\right)^{2\alpha-1}-1  & \mu_0\leq x<\mu_0\left[c\lambda t+1\right]^{\frac{1}{\alpha}}\\
		[c\lambda t+1]^{\frac{2\alpha-1}{\alpha}}-1	& x>\mu_0\left[c\lambda t+1\right]^{\frac{1}{\alpha}}
	\end{cases}\right]
	.\end{equation}

Figure \ref{f4} presents an 
example of this distribution, together with a histogram of simulation results.

To compute the variance of the sale price, we first note that for the Pareto distribution, assuming $\alpha>2$,
$$Var[X_0]=\frac{\alpha x_m^2}{(\alpha-1)^2(\alpha-2)},$$
and plug the above expressions into
\eqref{varf} to obtain
\begin{equation}\label{var_pareto}Var[S_t]=\frac{\alpha x_m^2}{(\alpha-2)(\alpha^2-1)}\left[
2 \left(\frac{\alpha}{\alpha-1}\right)[ c\lambda  t+1]^{\frac{2}{\alpha}}-\frac{1}{[ c\lambda  t+1]^{\frac{\alpha-1}{\alpha}} } \right].\end{equation}

\section{The time to sale}
\label{time_to_sale}

Another important random variable related to the 
bidding process is the time to sale, that is the stopping time, assuming the optimal policy is followed. For any marketing period $t\geq 0$, we will denote this random variable by $T_t$.
Obviously $T_t$ takes values in $[0,t]$.
It is of interest to evaluate the characteristics of $T_t$, e.g. its expectation and variance. In fact we will obtain the full distribution of $T_t$. 
We note that the distribution of $T_t$ has an
atom at the value $t$, that is $P(T_t=t)>0$ - this is the probability that a sale is not made before the deadline. 

It will also be of interest to examine the random variable
\begin{equation}\label{htt}\hat{T}_t=\frac{1}{t}T_t\in [0,1],\end{equation}
which signifies the proportion of 
the available marketing period that will in fact be used.

\subsection{Determining the stopping time distribution}

\begin{thm}\label{tts}
The cumulative distribution function of $T_t$ is given by
\begin{equation}\label{ts}H_t(r)=P(T_t\leq r)=\begin{cases}
		1-\frac{\phi(\mu(t))}{\phi(\mu(t-r))} & r<t\\
		1 & r\geq t
	\end{cases}.
\end{equation}
In particular, the 
probability that a sale is not made 
up to the deadline is given by
$$P(T_t=t)=H_t(t+)-H_t(t-)=\frac{\phi(\mu(t))}{\phi(\mu_0)}.$$
\end{thm}

\begin{proof}
The probability that a sale occurs during
the infinitesimal time interval $[s,s+ds]$ after the 
start of the bidding process is
$$\lambda [1-F(\mu(t-s))]ds,$$
where $\lambda ds$ is the probability that
an offer is received during this time interval, and $1-F(\mu(t-s))$
is the probability that such an offer is accepted. Therefore the probability that
a sale is {\it{not}} made up to time $r$ into the marketing period is
\begin{equation}\label{H}P(T_t\geq r)=\begin{cases}
		e^{-\lambda\int_0^r [1-F(\mu(t-s))]ds} & r<t\\
		1 & r\geq t.
	\end{cases}.\end{equation}
Using the identity \eqref{ii}, we have
$$\lambda\int_0^r [1-F(\mu(t-s))]ds=\lambda\int_{t-r}^t [1-F(\mu(s))]ds=
\ln\left(\frac{\phi(\mu(t-r))}{\phi(\mu(t))}\right),$$
which, together with \eqref{H}, implies \eqref{ts}.
\end{proof}

Once we have determined 
the distribution of $T_t$, we immediately obtain the distribution of
$\hat{T}_t$ given by \eqref{htt}:
\begin{equation}\label{htt1}\hat{H}_t(s)=P(\hat{T}_t\leq s)
=P(T_t\leq st)=H_t(st)
=\begin{cases}
	1-\frac{\phi(\mu(t))}{\phi(\mu((1-s)t))} & s<1\\
	1 & s\geq 1
\end{cases}.
\end{equation}

\subsection{Expectation and variance of the time to sale}

Using the expression \eqref{ts}, we can compute the expectation and variance of $T_t$.

\begin{thm}\label{evT} We have
\begin{equation}\label{ET}E[T_t]=\frac{1}{\lambda}\phi(\mu(t))\int_{\mu_0}^{\mu(t)}\frac{dw}{\phi(w)^2},\end{equation}
\begin{equation}\label{VT}
	Var[T_t]=\frac{1}{\lambda^2}\left[2\phi(\mu(t))\int_{\mu_0}^{\mu(t)}\frac{1}{\phi(u)}\int_{\mu_0}^{u}\frac{1}{\phi(w)^2}dwdu-\left[\phi(\mu(t))\int_{\mu_0}^{\mu(t)}\frac{dw}{\phi(w)^2}\right]^2\right].
\end{equation}
\end{thm}
\begin{proof}
We have, using \eqref{ts},
$$E[T_t]=\int_0^\infty P(T_t\geq r)dr= \int_0^t\frac{\phi(\mu(t))}{\phi(\mu(t-r))}dr
=\phi(\mu(t))\int_0^t\frac{1}{\phi(\mu(r))}dr.$$
Making the change of variable
\begin{equation}\label{cv}w=\mu(r),\;\; dw=\mu'(r)dr=\lambda \phi(\mu(r))dr=\lambda \phi(w)dr\end{equation}
in the last integral, we have
\begin{equation}\label{i1}
	\int_0^t\frac{1}{\phi(\mu(r))}dr=\frac{1}{\lambda}\int_{\mu_0}^{\mu(t)}\frac{1}{\phi(w)^2}dw,
\end{equation}
so we obtain
\eqref{ET}.

Using the identity \eqref{identity} and \eqref{ts}, we have
\begin{align}\label{ct}
	E[T_t^2]&=2\int_0^t r[1-H_t(r)]dr=2\phi(\mu(t))\int_0^t \frac{r}{\phi(\mu(t-r))}dr\\
	&=2\phi(\mu(t))\int_0^t \frac{t-r}{\phi(\mu(r))}dr=2\phi(\mu(t))\left[t\int_0^t \frac{1}{\phi(\mu(r))}dr-\int_0^t \frac{r}{\phi(\mu(r))}dr\right].\nonumber
\end{align}
Making the change of variable \eqref{cv} in the last integral, noting that
$$w=\mu(r)=\Psi^{-1}(\lambda r)\;\;\Rightarrow\;\;r=\frac{1}{\lambda}\Psi(w),$$
we get
\begin{align}\label{h1}\int_0^t \frac{r}{\phi(\mu(r))}dr&=\frac{1}{\lambda^2 }\int_{\mu_0}^{\mu(t)} \frac{\Psi(w)}{\phi(w)^2}dw=
	\frac{1}{\lambda^2}\int_{\mu_0}^{\mu(t)} \frac{1}{\phi(w)^2}\int_{\mu_0}^w \frac{du}{\phi(u)}dw
	=\frac{1}{\lambda^2}\int_{\mu_0}^{\mu(t)} \frac{1}{\phi(u)}\int_{u}^{\mu(t)}\frac{1}{\phi(w)^2} dwdu.\end{align}
Using
$$t=\frac{1}{\lambda}\Psi(\mu(t))=\frac{1}{\lambda}\int_{\mu_0}^{\mu(t)}\frac{du}{\phi(u)},$$
and \eqref{i1}, we can write
\begin{equation}\label{h2}
	t\int_0^t \frac{1}{\phi(\mu(r))}dr=\frac{1}{\lambda^2}\int_{\mu_0}^{\mu(t)}\frac{1}{\phi(u)}\int_{\mu_0}^{\mu(t)}\frac{1}{\phi(w)^2}dwdu.
\end{equation}
Combining \eqref{ct},\eqref{h1},\eqref{h2}, we get
$$E[T_t^2]=\frac{2\phi(\mu(t))}{\lambda^2}\int_{\mu_0}^{\mu(t)}\frac{1}{\phi(u)}\int_{\mu_0}^{u}\frac{1}{\phi(w)^2}dwdu,$$
which together with \eqref{ET}, give \eqref{VT}.
\end{proof}

\subsection{Examples}
\label{examples_T}

We will now calculate the 
distribution of $T_t$ for several examples. 
In each case, we also perform simulations 
of the bidding process, where we take the residual distribution to be identical with the offer distribution ($F_0=F$), and display 
a histogram of the stopping time which 
is compared with the analytical result, showing a perfect fit. Note that 
in this case we display cumulative 
histograms, because of the fact that 
the distribution of $T_t$ has an 
atom at $T_t=t$. The rightmost bar 
in each histogram contains the simulations
corresponding to $T_t=t$.

\subsubsection{Uniform offer distribution}

\begin{figure}
\centering
\includegraphics[height=6cm, angle=0]{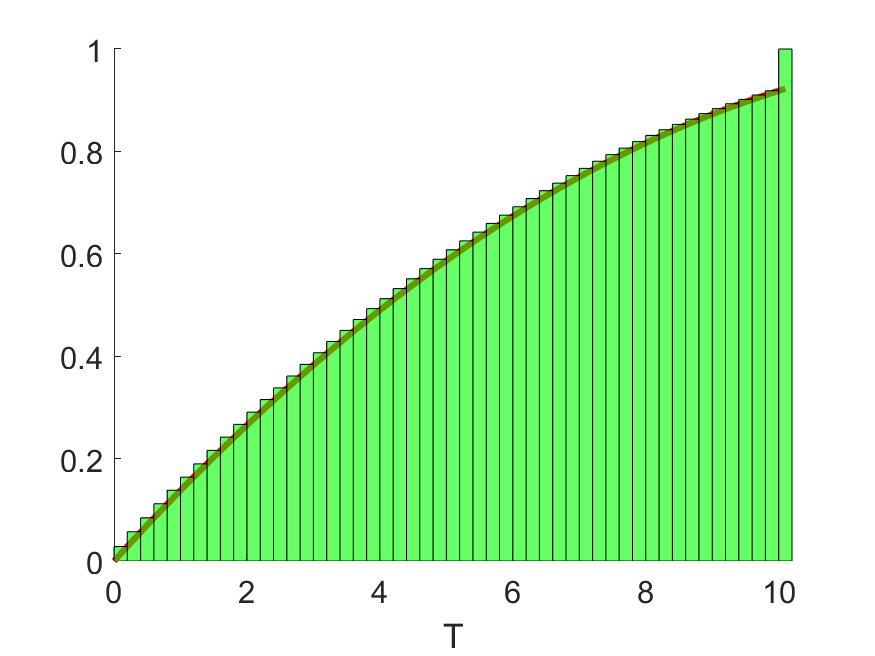}
\caption{Cumulative histogram of time to sale in $N=10^5$ simulations of the bidding process, employing the  optimal policy, where $\lambda=1,t=10$, and the offer and residual distributions are uniform on $[1,3]$. The red line shows the analytical expression for the cumulative density, given by \eqref{H_uniform}.}
\label{f5}
\end{figure}

We take a uniform offer distribution \eqref{uniform_denisity}, and allow the
residual distribution to be arbitrary, but
assume $\mu_0=E(X_0)$ satifies $\mu_0\in [a,b]$, so that 
\eqref{mu_uniform_g} is valid (if $\mu_0<a$ we would have to use \eqref{mu_uniform_g1} instead). 
It will be convenient to set
\begin{equation}\label{def_k}
k=\frac{2(b-a)}{b-\mu_0}
\end{equation}

We then have (using \eqref{ode})
$$\mu(t)=b-\frac{2(b-a)}{\lambda t+k},\;\;\;
\phi(\mu(t))=\frac{1}{\lambda}\mu'(t)=\frac{2(b-a)}{\left(\lambda t+k\right)^2},
$$
so that \eqref{ts} gives
\begin{equation}\label{H_uniform}H_t(r)=
\begin{cases}
	1-\left(1-\frac{\lambda r}{\lambda t+k}\right)^2 & r<t\\
	1 & r\geq t.
\end{cases}
\end{equation}

In particular the probability that a sale is not made before the deadline is
$$P(T_t=T)=H_t(t+)-H_t(t-)=\left(\frac{k}{\lambda t+k}\right)^2.$$
Figure \ref{f5} shows the graph of 
the CDF \eqref{H_uniform} 
when both the offer distribution and the residual distribution are uniform on $[1,3]$, so that $\mu_0=2$ and $k=4$, together with a histogram of the stopping time $T_t$ for $10^5$ simulations of the bidding process. 

Using Theorem \ref{evT}, we obtain

$$E[T_t]= \frac{1}{3\lambda }
\left[\lambda t+k - \frac{k^3}{\left(\lambda t+k\right)^2} \right]$$

$$Var[T_t]=\frac{\lambda t^3[(\lambda t)^3+6k(\lambda t)^2+15k^2\lambda t+12k^3]}{18(\lambda t+k)^4}$$




%
Note that 
\begin{equation}\label{atu}\lim_{t\rightarrow \infty}\frac{E[T_t]}{t}=\frac{1}{3},\;\;\;\lim_{t\rightarrow \infty}\frac{Var[T_t]}{t^2}=\frac{1}{18}.\end{equation}
The same asymptotic limits are obtained in \cite{Mazalov} for the discrete-time problem with a uniform offer distribution, in which case explicit 
expressions for $E[T_t]$, $Var[T_t]$ are not available.

Using \eqref{htt1}, we have that the distribution of 
$\hat{T}_t$ is given by
$$\hat{H}_t(s)=P(\hat{T}_t\leq s)
=	\begin{cases}
1-\left(1-\frac{\lambda ts }{\lambda t+k}\right)^2 & s<1\\
1 & s\geq 1.
\end{cases}
$$
hence, as $t\rightarrow \infty$, the distribution of  $\hat{T}_t$ converges to
\begin{equation}\label{limH1}H_\infty(s)=\lim_{t\rightarrow\infty}\hat{H}_t(s)=\begin{cases}
	2s-s^2 & s<1\\
	1 & s\geq 1
\end{cases}.\end{equation}

\subsubsection{Exponential offer distribution}

\begin{figure}
\centering
\includegraphics[height=6cm, angle=0]{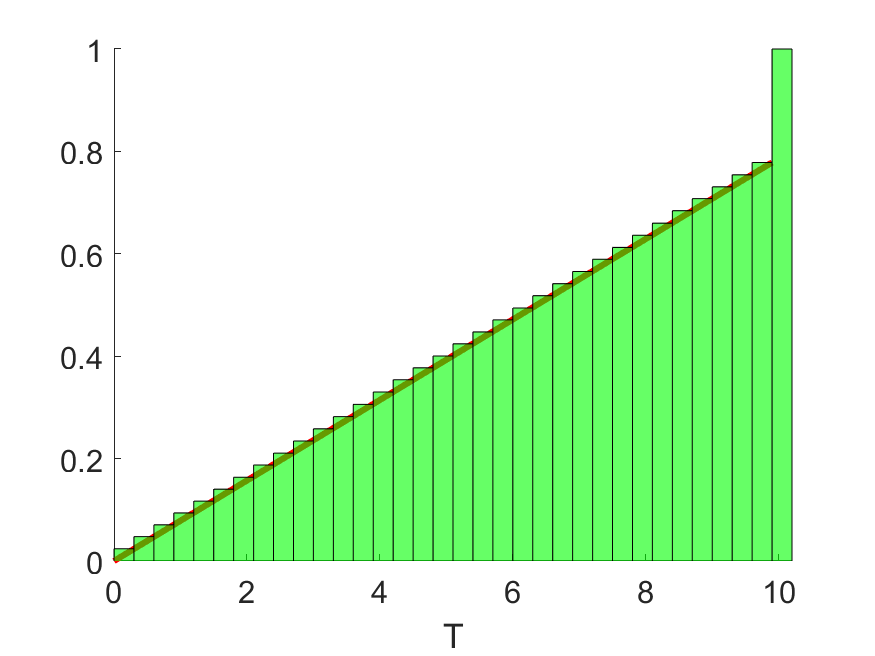}
\caption{Cumulative histogram of time to sale in $N=10^5$ simulations of the bidding process, employing the  optimal policy, where $\lambda=1,t=10$, and the offer and residual distributions are exponential with mean $\eta=2$. The red line shows the analytical expression for the cumulative density, given by \eqref{H_exponential}.}
\label{f6}
\end{figure}

We consider an exponential offer distribution \eqref{exp},
and an arbitrary residual distribution with $\mu_0=E[X_0]\geq 0$.
Using \eqref{mu_exp} we have
$$\phi(\mu(t))=\frac{1}{\lambda}\mu'(t)=\frac{\eta}{\lambda t+e^{\frac{\mu_0}{\eta}}},$$
so, by \eqref{ts}
\begin{equation}\label{H_exponential}H_t(r)=\begin{cases}\frac{\lambda r}{\lambda t+e^{\frac{\mu_0}{\eta}}}
	& r<t\\
	1& r\geq t.
\end{cases}
\end{equation}
The probability that a sale is not made before the deadline is
$$P(T_t=t)=H_t(t+)-H_t(t-)=\frac{e^{\frac{\mu_0}{\eta}}}{\lambda t+e^{\frac{\mu_0}{\eta}}}
.$$
The fact that $H_t(r)$ is linear in $r$ means that,
conditional on a sale being made before the deadline, the time at which the sale is made is distributed uniformly on $[0,t]$. 
Using Theorem \ref{evT}, we compute:
\begin{equation}\label{texp}E[T_t]=\frac{1}{2}t\cdot \left[1+\frac{e^{\frac{\mu_0}{\eta}} }{\lambda t+e^{\frac{\mu_0}{\eta}}} \right],\;\;\;Var[T_t]=\frac{\lambda t^3(\lambda t+4e^{\frac{\mu_0}{\eta}})}{12(\lambda t+e^{\frac{\mu_0}{\eta}})^2}.
\end{equation}
\eqref{texp} implies
\begin{equation}\label{texpa}\lim_{t\rightarrow \infty}\frac{E[T_t]}{t}=\frac{1}{2},\;\;\;\lim_{t\rightarrow \infty}\frac{Var[T_t]}{t^2}=\frac{1}{12},
\end{equation}
so that when the marketing period is sufficiently long, the expected time to sale will be
about half of the available time.
We note that the same asymptotic results as in \eqref{texpa} are obtained in the exponential distribution case for the discrete-time version of the problem in \cite{entwistle} (Table 2).

Using \eqref{htt1}, we have that the distribution of 
$\hat{T}_t$ is given by
$$\hat{H}_t(s)=P(\hat{T}_t\leq s)=\begin{cases}
\frac{\lambda t}{\lambda t +e^{\frac{\mu_0}{\eta}}}\cdot s & s<1\\
1 & s\geq 1,
\end{cases}$$
hence, as $t\rightarrow \infty$, the distribution of  $\hat{T}_t$ converges to a uniform distribution on $[0,1]$, 
\begin{equation}\label{limH2}\lim_{t\rightarrow\infty}\hat{H}_t(s)=\begin{cases}
	s & s<1\\
	1 & s\geq 1
\end{cases}.\end{equation}

\subsubsection{Pareto offer distribution}

\begin{figure}
\centering
\includegraphics[height=6cm, angle=0]{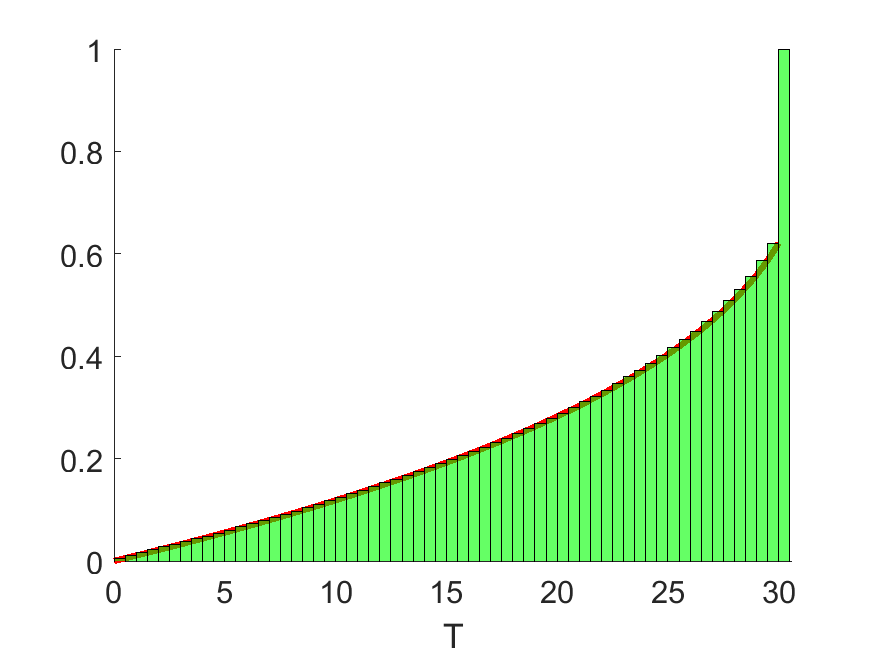}
\caption{Cumulative histogram of time to sale in $N=10^5$ simulations of the bidding process, employing the  optimal policy, where $\lambda=1,t=30$, and the offer distribution and residual distributions are Pareto, with $x_m=1,\alpha=1.5$. The red line shows the analytical expression for the cumulative density, given by \eqref{H_pareto}.}
\label{f7}
\end{figure}

Assuming  a Pareto offer distribution \eqref{pareto}, 
and an arbitrary residual distribution with $\mu_0=E[X_0]\geq x_m$, we have,
using \eqref{cp},\eqref{ts},
\begin{equation}\label{H_pareto}H_t(r)=\begin{cases}1-\left(1-\frac{c\lambda r}{ c\lambda t+1}\right)^{^{1-\frac{1}{\alpha}}} & r<t\\
	1 & r\geq t.
	\end{cases}\end{equation}
	where $c$ is given by \eqref{defc}.
	In particular the probability that 
	a sale is not made before the deadline is
	$$P(T_t=t)=H_t(t+)-H_t(t-)=\left(1-\frac{c\lambda r}{ c\lambda t+1}\right)^{^{1-\frac{1}{\alpha}}}.$$
	Using Theorem \ref{evT}, the expectation and variance of $T_t$ are
	
	%
	%
	%
	%
	%
	%
	%

	$$E[T_t]=\frac{1}{\lambda }\frac{\alpha-1}{2\alpha-1} \left(\frac{\mu_0}{x_m}\right)^\alpha\left[c\lambda t+1\right]\left[1-\left[c\lambda t+1\right]^{\frac{1}{\alpha}-2} \right]$$

	\begin{align*}Var[T_t]&=\frac{1}{\lambda^2}\frac{(\alpha-1)^2}{(2\alpha-1)}\left(\frac{\mu_0}{x_m} \right)^{2\alpha} [c\lambda t+1]^2\left(2 \left[\frac{1}{3\alpha-1}\left(1-\frac{1}{\left[c\lambda t+1 \right]^{3-\frac{1}{\alpha}}}\right)-\frac{1}{\alpha}\frac{c\lambda t}{\left[c\lambda t+1 \right]^{3-\frac{1}{\alpha}}} \right]\right.\\
&\left.	-\frac{1}{2\alpha-1}\left[1-\frac{1}{\left[c \lambda t+1\right]^{2-\frac{1}{\alpha}}} \right]^2\right).
\end{align*}



From these results we have
\begin{equation}\label{as_pareto}\lim_{t\rightarrow \infty}\frac{E[T_t]}{t}= \frac{\alpha}{2\alpha-1},\;\;\;\lim_{t\rightarrow \infty}\frac{Var[T_t]}{t^2}=\frac{\alpha^2(\alpha-1)}{(2\alpha-1)^2(3\alpha-1)}.
\end{equation}

We note that \eqref{as_pareto} is analogous to the asymptotic results for the time to sale obtained for the Pareto distribution in the discrete-time version of the problem in \cite{entwistle}
(Table 2), in which case explicit formulas for $E[T_t]$, $Var[T_t]$ are not 
available.

Using \eqref{htt1}, we have that the distribution of 
$\hat{T}_t$ is given by
$$\hat{H}_t(s)=P(\hat{T}_t\leq s)=\begin{cases}
1-\left(1-\frac{c \lambda t }{ c  \lambda t+1}\cdot s \right)^{^{\frac{1}{c}}} & s<1\\
1 & s\geq 1,
\end{cases}$$
hence, as $t\rightarrow \infty$, the distribution of  $\hat{T}_t$ converges to 
\begin{equation}\label{limH3}\lim_{t\rightarrow\infty}\hat{H}_t(s)=\begin{cases}
1-(1-s)^{1-\frac{1}{\alpha}} & s<1\\
1 & s\geq 1
\end{cases}.\end{equation}

\section{Long marketing period - asymptotic results}
\label{asy}

In this final section we 
exploit results obtained in previous sections in order to study the 
asymptotic behavior of some key 
quantities when the marketing period $t$ is long, $t\rightarrow \infty$. In subsection 
\ref{asymptotics1} we consider 
$\mu(t)=E(S_t)$, and $Var[S_t]$, and 
in section \ref{T_large} we study the limiting distribution of the random variable $\hat{T}_t=\frac{1}{t}T_t$, the fraction of the marketing period until a sale is made.

\subsection{Asymptotics of mean and variance of $S_t$}
\label{asymptotics1}

The next results describe
the asymptotic behavior of $E(S_t)=\mu(t)$ and $Var[S_t]$ for $t$ large, for three classes of offer distributions, defined 
by their tail behavior.
The results concerning $\mu(t)$ are 
analogous to the results of \cite{Kennedy} (see also \cite{Livanos} for related results), which analyzed the 
asymptotics of the threshold sequence 
$\mu_n$ in the discrete-time case.
The fact that in the continuous-time case we have explicit expressions for the relevant quantities allows 
the asymptotic analysis to be carried out by elementary tools.

\subsubsection{Offer distributions with  support bounded from above}

We now consider the case of offer distributions with $M<\infty$, where $M$ is given by \eqref{def_M}.
In this case, by \eqref{ll}, when
the available time is long ($t\rightarrow\infty$), we have $\mu(t)=E[S_t]\rightarrow M$. The following result characterizes the rate of convergence, showing that it is related to the behavior of the CDF $F(x)$ near the right
endpoint $M$.

The notation $\sim$ means that the ratio of the two sides approaches $1$, as $t\rightarrow \infty$. 

\begin{thm}\label{tep}
Assume 
$M<\infty$, where $M$ is given by \eqref{def_M}.
Assume also that, for some $p>0,c>0$, 
\begin{equation}\label{edge}\lim_{x\rightarrow M-}\frac{\phi(x)}{(M-x)^{p+1}}=c.\end{equation}
Then we have, as $t\rightarrow +\infty$
\begin{equation}\label{ba}\mu(t)=M-\frac{1}{(\lambda pc)^{\frac{1}{p} }}\cdot \frac{1}{t^{\frac{1}{p}}}+o\left( \frac{1}{t^{\frac{1}{p}}}\right) .\end{equation}
The asymptotic behavior of $Var[S_t]$ depends on $p$ as follows:

If $p<1$ then
\begin{equation}\label{p<1}
Var[S_t]\sim \frac{1}{(\lambda p)^{\frac{p+1}{p} }c^{\frac{1}{p}}}\left[(M-\mu_0)+\frac{Var[X_0]}{\phi(\mu_0)}+\int_{\mu_0}^{M}\frac{Var[(X-w)_+]}{E[(X-w)_+]^2}  dw \right]\cdot \frac{1}{t^{\frac{p+1}{p}}}.
\end{equation}
If $p=1$ then
\begin{equation}\label{p=1}
Var[S_t]\sim\frac{2}{3\lambda^2 c^2}\cdot \frac{\ln(t)}{t^2}.
\end{equation}
If $p>1$ then
\begin{equation}\label{p>1}
Var[S_t]\sim\frac{2}{p-1}\cdot \frac{1}{p+2}\cdot \frac{1}{ (\lambda pc)^{\frac{2}{p}}}\cdot \frac{1}{t^{\frac{2}{p}}}.
\end{equation}

\end{thm}

\begin{proof}

Recalling that $\Psi'(x)=\frac{1}{\phi(x)}$, we have
by \eqref{psi_inf}, L'H\^opital's rule, and \eqref{edge}
\begin{equation}\label{llm}
\lim_{x\rightarrow M-}\frac{\Psi(x)}{(M-x)^{-p}}=\lim_{x\rightarrow M-}\frac{\Psi'(x)}{p(M-x)^{-p-1}}
=\lim_{x\rightarrow M-}\frac{(M-x)^{p+1}}{p\phi(x)}=\frac{1}{pc}.
\end{equation}
Setting $x=\mu(t)$ in \eqref{llm}, and using the fact that 
$\Psi(\mu(t))=\lambda t$, we get
\begin{align*}&\lim_{t\rightarrow +\infty}\lambda t (M-\mu(t))^{p} =\lim_{t\rightarrow +\infty}\frac{\Psi(\mu(t))}{(M-\mu(t))^{-p}}=\frac{1}{pc}\;\;\Rightarrow\;\;(M-\mu(t))^{p} =\frac{1}{\lambda pc}\frac{1}{t}(1+o(1))\\
&\Rightarrow\;\;M-\mu(t) =\frac{1}{(\lambda pc)^{\frac{1}{p}}}\cdot\frac{1}{t^{\frac{1}{p}}}\left(1+o\left(1 \right)\right)^{\frac{1}{p}}=\frac{1}{(\lambda pc)^{\frac{1}{p}}}\cdot\frac{1}{t^{\frac{1}{p}}}\left(1+o\left(1 \right) \right),
\end{align*}
which implies \eqref{ba}.

We now turn to analyzing $Var[S_t]$.
We have $\phi(x)=0$ for 
$x>M$, hence
defining 
$$g(t)=\int_{\mu_0}^{\mu(t)}\frac{1}{\phi(w)^2} \int_{w}^M \phi(x) dx dw,$$
we have, using \eqref{varf},
\begin{equation}\label{varb}Var[S_t]=\phi(\mu(t)) \left[\frac{Var[X_0]}{\phi(\mu_0)} +2g(t)
\right]\end{equation}
Substituting $x=\mu(t)$ in \eqref{edge}
we have
\begin{equation}\label{ed1}
\lim_{t\rightarrow \infty}\frac{\phi(\mu(t))}{(M-\mu(t))^{p+1}}=c.
\end{equation}
From \eqref{ba} we have
\begin{equation}\label{ed0}\lim_{t\rightarrow \infty}\frac{M-\mu(t)}{t^{-\frac{1}{p}}}=\frac{1}{(\lambda pc)^{\frac{1}{p} }},\end{equation}
which implies
\begin{equation}\label{ed2}\lim_{t\rightarrow \infty}\frac{(M-\mu(t))^{p+1}}{t^{-\frac{p+1}{p}}}=\frac{1}{(\lambda pc)^{\frac{p+1}{p} }}.\end{equation}
Combining \eqref{ed1} and \eqref{ed2} we have
\begin{equation}\label{ed3}
\lim_{t\rightarrow \infty}\frac{\phi(\mu(t))}{t^{-\frac{p+1}{p}}}=\frac{c}{(\lambda pc)^{\frac{p+1}{p} }}=\frac{1}{(\lambda p)^{\frac{p+1}{p} }c^{\frac{1}{p}}}.
\end{equation}
We have
\begin{equation}\label{rr}g'(t)=\frac{\mu'(t)}{\phi(\mu(t))^2}\int_{\mu(t)}^M\phi(x)dx=\frac{\lambda}{\phi(\mu(t))}\int_{\mu(t)}^M\phi(x)dx.\end{equation}

Using the integral mean-value theorem we have
\begin{align*}\int_v^M \phi(x)dx&= \int_v^M \frac{\phi(x)}{(M-x)^{p+1}}\cdot (M-x)^{p+1} dx= \frac{\phi(\tilde{x})}{(M-\tilde{x})^{p+1}}\cdot\int_v^M  (M-x)^{p+1} dx\\
&=\frac{\phi(\tilde{x})}{(M-\tilde{x})^{p+1}}\cdot \frac{(M-v)^{p+2}}{p+2}\;\;
\Rightarrow\;\; \frac{1}{(M-v)^{p+2}}\int_v^M \phi(x)dx =\frac{1}{p+2}\cdot\frac{\phi(\tilde{x})}{(M-\tilde{x})^{p+1}},
\end{align*}
where $\tilde{x}\in (v,M)$, which together with \eqref{edge} implies 
\begin{equation}\label{em1}\lim_{v\rightarrow M-}\frac{1}{(M-v)^{p+2}}\int_v^M \phi(x)dx =\frac{c}{p+2}.\end{equation}
\eqref{edge} together with \eqref{em1} give
\begin{equation*}\lim_{v\rightarrow M-}\frac{1}{M-v}\cdot\frac{1}{\phi(v)}\int_{v}^M \phi(x)dx=\frac{1}{p+2},\end{equation*}
and substituting $v=\mu(t)$ we get
$$\lim_{t\rightarrow\infty}\frac{1}{M-\mu(t)}\cdot \frac{1}{\phi(\mu(t))}\int_{\mu(t)}^M \phi(x)dx=\frac{1}{p+2},$$
which, combined with \eqref{ed0}, gives
\begin{equation}\label{ed6}\lim_{t\rightarrow\infty}t^{\frac{1}{p}}\cdot \frac{1}{\phi(\mu(t))}\int_{\mu(t)}^M \phi(x)dx=\frac{1}{p+2}\cdot \frac{1}{(\lambda p c)^{\frac{1}{p}}}.\end{equation}
From \eqref{rr} and \eqref{ed6} we have
\begin{equation}\label{kp}\lim_{t\rightarrow\infty}t^{\frac{1}{p}}g'(t)=\frac{1}{p+2}\cdot \frac{\lambda}{(\lambda p c)^{\frac{1}{p}}}=\frac{1}{p+2}\cdot \frac{1}{(p c)^{\frac{1}{p}}\lambda^{\frac{1}{p}-1}}.\end{equation}
If $p<1$, \eqref{kp} implies that the integral $\int_{\mu_0}^\infty g'(t)dt$ is finite, so that
$$g(\infty)=\lim_{t\rightarrow +\infty} g(t)=\int_{\mu_0}^{\infty}\frac{1}{\phi(w)^2} \int_{w}^M \phi(x) dx dw<\infty,$$
and together with \eqref{varb},\eqref{ed3} we 
get
\begin{equation}\label{star}Var[S_t]=\frac{1}{(\lambda p)^{\frac{p+1}{p} }c^{\frac{1}{p}}}\left[\frac{Var[X_0]}{\phi(\mu_0)}+2g(\infty) \right]\cdot \frac{1}{t^{\frac{p+1}{p}}}(1+o(1)).
\end{equation}
Note that by the identity
\eqref{idk} we have
$$2g(\infty)=(M-\mu_0)+\int_{\mu_0}^{M}\frac{Var[(X-w)_+]}{E[(X-w)_+]^2}  dw,$$
so \eqref{star} implies \eqref{p<1}.

If $p>1$, then, using L'H\^opital's rule, \eqref{kp} implies
\begin{equation}\label{af}\lim_{t\rightarrow\infty}\frac{g(t)}{t^{\frac{p-1}{p}}}=\lim_{t\rightarrow\infty}\frac{g'(t)}{\frac{p-1}{p}\cdot t^{-\frac{1}{p}}}
=\frac{p}{p-1}\cdot \frac{1}{p+2}\cdot \frac{1}{(p c)^{\frac{1}{p}}\lambda^{\frac{1}{p}-1}}.
\end{equation}
Combining \eqref{ed3} and \eqref{af}, we get
$$ \phi(\mu(t))g(t)= \frac{1}{p-1}\cdot \frac{1}{p+2}\cdot \frac{1}{(\lambda pc)^{\frac{2}{p}}}\cdot \frac{1}{t^{\frac{2}{p}}} (1+o(1)),\;\;{\mbox{as}}\;\; t\rightarrow \infty.$$

From \eqref{ed3} we have
$$\phi(\mu(t))\frac{Var[X_0]}{\phi(\mu_0)}=\frac{Var[X_0]}{\phi(\mu_0)}\frac{1}{(\lambda p)^{\frac{p+1}{p} }c^{\frac{1}{p}}}\cdot \frac{1}{t^{\frac{p+1}{p}}}(1+o(1)),\;\;{\mbox{as}}\;\; t\rightarrow \infty,$$
and since $\frac{p+1}{p}>\frac{2}{p}$
we conclude that
$$Var[S_t]=\frac{2}{p-1}\cdot \frac{1}{p+2}\cdot \frac{1}{ (\lambda pc)^{\frac{2}{p}}}\cdot \frac{1}{t^{\frac{2}{p}}} (1+o(1)),\;\;{\mbox{as}}\;\; t\rightarrow \infty,$$
so we have \eqref{p>1}.

When $p=1$ we have, using  L'H\^opital's rule and \eqref{kp},
\begin{equation}\label{af1}\lim_{t\rightarrow\infty}\frac{g(t)}{\ln(t)}=\lim_{t\rightarrow\infty}\frac{g'(t)}{\frac{1}{t}}
=\frac{1}{3 c},
\end{equation}
and from \eqref{ed3} we have
$$\lim_{t\rightarrow \infty} t^2\phi(\mu(t))=\frac{1}{\lambda ^{2}c},$$
implying
$$Var[S_t]=\frac{2}{3\lambda^2 c^2}\cdot \frac{\ln(t)}{t^2}(1+o(1)),\;\;{\mbox{as}}\;\; t\rightarrow \infty.$$

\end{proof}

As an example for the application of the above theorem, for the Beta distribution on $[0,1]$,  \begin{equation}\label{Beta}
f(x)=\frac{\Gamma(\alpha+\beta)}{\Gamma(\alpha)\Gamma(\beta)}x^{\alpha-1}(1-x)^{\beta-1},\;\;\;\alpha,\beta>0,
\end{equation}
using L'H\^opital's rule we have
$$\lim_{x\rightarrow 1-}\frac{\phi(x)}{(1-x)^{\beta+1}}=-
\lim_{x\rightarrow 1-}\frac{F(c)-1}{(\beta+1)(1-x)^{\beta}}
=\lim_{x\rightarrow 1-}\frac{f(x)}{(\beta+1)\beta(1-x)^{\beta-1}}=\frac{\Gamma(\alpha+\beta)}{\Gamma(\alpha)\Gamma(\beta+2)},
$$
so we have \eqref{edge} with 
$M=1,p=\beta,c=\frac{\Gamma(\alpha+\beta)}{\Gamma(\alpha)\Gamma(\beta+2)}$,
so \eqref{ba} gives
\begin{equation}\label{beta}\mu(t)=M-\left(\frac{\Gamma(\alpha)\Gamma(\beta+2)}{\Gamma(\alpha+\beta)}\cdot \frac{1}{\beta \lambda } \right)^{\frac{1}{\beta}}\cdot \frac{1}{t^{\frac{1}{\beta}}}+o\left(\frac{1}{t^{\frac{1}{\beta}}} \right),\end{equation}
and the asymptotics of $Var[S_t]$ can also be obtained by the results above, depending on the value of $\beta$.

In particular, when $\alpha=\beta=1$ we get the uniform distribution on $[0,1]$, , and we have
$c=\frac{1}{2}$, so\eqref{beta} reduces to 
\eqref{ba0} with $a=0,b=1$, and \eqref{p=1} gives
$Var[S_t]\sim \frac{8}{3\lambda^2}\cdot \frac{\ln(t)}{t^2},$
which is consistent with \eqref{uni_var}.

\subsubsection{Offer distributions with an exponential right tail}
\label{ofexp}

We now consider the case in which 
the right-tail behavior of the offer distribution is exponential, in the sense made precise by \eqref{exptail} below.

\begin{thm}\label{thm:exptail}
Assume that $M=+\infty$ and the function $\phi(x)$ defined by \eqref{defphi} satisfies 
\begin{equation}\label{exptail}
\lim_{x\rightarrow \infty} \frac{\phi(x)}{\int_{x}^\infty \phi(u)du}=c,\;\;\;c>0.
\end{equation}
Then, as $t\rightarrow \infty$,
\begin{equation}\label{expc}
\mu(t)\sim\frac{1}{c}\ln(t),\;\;\;\mu'(t)\sim \frac{1}{ct},
\end{equation}
\begin{equation}\label{expv}
\lim_{t\rightarrow \infty} Var[S_t]=\frac{2}{c^2}.
\end{equation}
\end{thm}

\begin{proof}
Define
$$I(x)=\int_x^\infty \phi(u)du.$$	
By \eqref{exptail} we have
$$\lim_{x\rightarrow \infty} \frac{I'(x)}{I(x)}=-c.$$
Therefore, fixing any $\epsilon\in (0,c)$, we can chose 
$x_0>0$ so that 
\begin{equation}\label{limiz}
x\geq  x_0\;\;\Rightarrow\;\; -c-\epsilon\leq \frac{I'(x)}{I(x)}\leq -c+\epsilon. \end{equation}
Assume $x\geq u\geq x_0$. using \eqref{limiz}, we have
\begin{equation}\label{kin}
\ln\left(\frac{I(x)}{I(u)}\right)=\int_{u}^{x} \frac{I'(v)}{I(v)}dv \leq (-c+\epsilon)(x-u)\;\;\Rightarrow\;\;I(u)\geq I(x)e^{(c-\epsilon)(x-u)}. \end{equation}
Using \eqref{limiz} and \eqref{kin} we get
$$\phi(u)=-I'(u)\geq -(-c+\epsilon)I(u)\geq (c-\epsilon) I(x) e^{(c-\epsilon)(x-u)}\;\;\Rightarrow\;\;\frac{1}{\phi(u)}\leq \frac{1}{(c-\epsilon)I(x)}e^{-(c-\epsilon)(x-u)},$$
hence, for $x>x_0$,
\begin{equation}\label{inti}\int_{x_0}^{x} \frac{1}{\phi(u)}du\leq \frac{1}{(c-\epsilon)I(x)}\int_{x_0}^{x} e^{-(c-\epsilon)(x-u)}du
=\frac{1-e^{-(c-\epsilon)(x-x_0)}}{(c-\epsilon)^2I(x)}.\end{equation}
By \eqref{limiz} and the fact that $I'(x)<0$ we have
\begin{equation}\label{si}\frac{1}{I(x)}\leq -\frac{c+\epsilon}{I'(x)}=\frac{c+\epsilon}{\phi(x)},\end{equation}
and combining \eqref{inti} and \eqref{si} gives, for $x>x_0$,
$$\int_{x_0}^{x} \frac{1}{\phi(u)}du\leq \frac{1-e^{-(c-\epsilon)(x-x_0)}}{(c-\epsilon)^2}\frac{c+\epsilon}{\phi(x)}\;\;\Rightarrow\;\; \phi(x)\int_{x_0}^{x} \frac{1}{\phi(u)}du\leq \frac{(1-e^{-(c-\epsilon)(x-x_0)})(c+\epsilon)}{(c-\epsilon)^2}.$$
We conclude that
$$\limsup_{x\rightarrow \infty}\phi(x)\int_{x_0}^{x} \frac{1}{\phi(u)}du\leq \frac{c+\epsilon}{(c-\epsilon)^2}.$$
Since $\lim_{x\rightarrow \infty}\phi(x)=0$, this also implies
$$\limsup_{x\rightarrow \infty}\phi(x)\int_{\mu_0}^{x} \frac{1}{\phi(u)}du=\limsup_{x\rightarrow \infty}\left[\phi(x)\int_{\mu_0}^{x_0} \frac{1}{\phi(u)}du+\phi(x)\int_{x_0}^{\infty} \frac{1}{\phi(u)}du \right]\leq \frac{c+\epsilon}{(c-\epsilon)^2},$$
and since $\epsilon$ is arbitrary we conclude that
$$\limsup_{x\rightarrow \infty}\phi(x)\int_{\mu_0}^{x} \frac{1}{\phi(u)}du\leq \frac{1}{c}.$$
An analogous argument, using \eqref{limiz}, gives 
$$\liminf_{x\rightarrow \infty}\phi(x)\int_{\mu_0}^{x} \frac{1}{\phi(u)}du\geq \frac{1}{c},$$
so that we have
\begin{equation}\label{il}\lim_{x\rightarrow \infty}\phi(x)\Psi(x)=\lim_{x\rightarrow \infty}\phi(x)\int_{\mu_0}^{x} \frac{1}{\phi(u)}du= \frac{1}{c}.\end{equation}
Note in particular that \eqref{il} implies
$$\lim_{x\rightarrow \infty}\Psi(x)=+\infty\;\;\Rightarrow\;\; \lim_{t\rightarrow \infty}\mu(t)=\lim_{t\rightarrow \infty}\Psi^{-1}(\lambda t)=+\infty.$$
Substituting $x=\mu(t)$ in \eqref{il} and using \eqref{ode} and \eqref{ef} we get 
$$\lim_{t\rightarrow \infty}\mu'(t)t=\lim_{t\rightarrow \infty} \phi(\mu(t))\Psi(\mu(t))=\frac{1}{c}$$
and using L'H\^opital's rule we obtain
$$\lim_{t\rightarrow \infty} \frac{\mu(t)}{\ln(t)}=\lim_{t\rightarrow \infty}\frac{\mu(t)}{\frac{1}{t}}=\frac{1}{c},$$
so we have \eqref{expc}.

We now prove \eqref{expv}.
We have, using \eqref{exptail} and L'H\^opital's rule
\begin{align}\label{texpp}&\lim_{t\rightarrow \infty}\frac{1}{t}\int_{\mu_0}^{\mu(t)}\frac{1}{\phi(w)^2} \int_{w}^\infty \phi(x) dx dw=\lim_{t\rightarrow \infty}\frac{d}{dt}\int_{\mu_0}^{\mu(t)}\frac{1}{\phi(w)^2} \int_{w}^\infty \phi(x) dx dw\nonumber\\&=\lim_{t\rightarrow \infty}\mu'(t)\frac{1}{\phi(\mu(t))^2} \int_{\mu(t)}^\infty \phi(x) dx =\lim_{t\rightarrow \infty}\frac{\lambda \int_{\mu(t)}^\infty \phi(x) dx  }{\phi(\mu(t))} 
=\lim_{u\rightarrow \infty}\frac{\lambda \int_{u}^\infty \phi(x) dx  }{\phi(u)}=\frac{\lambda}{c}.
\end{align}
From \eqref{ode} and \eqref{expc}
we have 
\begin{equation}\label{texp1}\lim_{t\rightarrow \infty}t\phi(\mu(t))=\frac{1}{\lambda}\lim_{t\rightarrow \infty} t\mu'(t)=\frac{1}{\lambda c}.\end{equation}
Using \eqref{varf} and \eqref{texp1} we get
$$\lim_{t\rightarrow \infty}Var[S_t]=\lim_{t\rightarrow \infty}\phi(\mu(t)) \frac{Var[X_0]}{\phi(\mu_0)} +2 \lim_{t\rightarrow \infty}t\phi(\mu(t))\cdot \lim_{t\rightarrow \infty}\frac{1}{t}
\int_{\mu_0}^{\mu(t)}\frac{1}{\phi(w)^2} \int_{w}^\infty \phi(x) dx dw=0+2\frac{1}{\lambda c}\cdot \frac{\lambda}{ c}=\frac{2}{c^2}.$$

\end{proof}

As an example for the application of the above result, consider the Gamma distribution, with density
\begin{equation}\label{gammad}f(x)=\frac{1}{\Gamma(\alpha)\eta^\alpha}\cdot x^{\alpha-1}e^{-\frac{x}{\eta}},\;\;\;\alpha>0,\eta>0.\end{equation}
We have 
$\phi''(x)=[F(x)-1]'=f(x)$,
hence, using L`H\^opital's rule repeatedly
\begin{align*}\lim_{x\rightarrow \infty} \frac{\phi(x)}{\int_{x}^\infty \phi(u)du}&=-\lim_{x\rightarrow\infty}\frac{\phi'(x)}{\phi(x)}=-\lim_{x\rightarrow\infty}\frac{\phi''(x)}{\phi'(x)}=-\lim_{x\rightarrow\infty}\frac{\phi'''(x)}{\phi''(x)}=-\lim_{x\rightarrow\infty}\frac{f'(x)}{f(x)}\\&=-\lim_{x\rightarrow\infty}\frac{\left((\alpha-1)x^{\alpha-2}e-\frac{1}{\eta}x^{\alpha-1}\right)e^{-\frac{x}{\eta}}}{x^{\alpha-1}e^{-\frac{x}{\eta}}}=\frac{1}{\eta},\end{align*}
so \eqref{exptail} holds with $c=\frac{1}{\eta}$, and \eqref{expc},\eqref{expv} give
$$\mu(t)\sim \eta\ln(t),\;\;\lim_{t\rightarrow\infty}Var[S_t]=2\eta^2.$$
In particular, when $\alpha=1$ we get the exponential distribution, and the result is consistent with the exact expressions
\eqref{mu_exp}, \eqref{var_expo}.

\subsubsection{Offer distributions with power-law tails}

We now consider the case in which the offer distribution $F(x)$ has a power-law right tail behavior, in the sense made precise in the theorem below.

We will use some results from the theory of regularly varying functions. We recall \cite{Mladenovic} that a measurable function $R:[0,\infty)\rightarrow [0,\infty)$ is 
said to be {\it{regularly varying}} with exponent $\rho$, if 
$$\lim_{s\rightarrow \infty}\frac{R(xs)}{R(s)}=x^{\rho},\;\;\;\forall x>0.$$

\begin{thm}\label{pow}
Suppose $M=+\infty$ and the function $\phi(x)$ satisfies
\begin{equation}\label{alg}
\lim_{x\rightarrow \infty} \frac{x\phi(x)}{\int_x^{\infty}\phi(u)du}=p,\;\;\;p>0.
\end{equation}	
Then $\mu(t)$ is regularly varying with exponent $\frac{1}{p+2}$, and, in particular, we have
\begin{equation}\label{limi}\lim_{t\rightarrow \infty} \frac{\mu(t)}{t^c}=\begin{cases}
	+\infty & c< \frac{1}{p+2}\\
	0 & c>\frac{1}{p+2}\end{cases}.\end{equation}
	$Var[S_t]$, as a function of $t$ is regularly varying with exponent $\frac{2}{p+2}$. In particular, we have
	$$\lim_{t\rightarrow \infty}\frac{Var[S_t]}{t^c}=\begin{cases}
+\infty & c<\frac{2}{p+2}\\
0 & c>\frac{2}{p+2}
\end{cases}.$$
\end{thm}

\begin{proof}
By \cite{Mladenovic}, Theorem 1.2.1(c), the condition \eqref{alg} implies that
$\phi\in RV_{-p-1}$.
This implies that $\frac{1}{\phi}\in RV_{p+1}$, and by  \cite{Mladenovic}, Theorem 1.2.1(a) we have that $\Psi$ defined by 
\eqref{def_Psi} satisfies $\Psi\in RV_{p+2}$. Since $\Psi$ is increasing and $\Psi(+\infty)=+\infty$, 
we conclude by \cite{Mladenovic}, Theorem 1.4.7 that $\Psi^{-1}\in RV_{\frac{1}{p+2}}.$
This implies that $\mu(t)=\Psi^{-1}(\lambda t)\in RV_{\frac{1}{p+2}}$, as we wanted to prove. Therefore
$\frac{\mu(t)}{x^c}\in RV_{\frac{1}{p+2}-c}$, which implies \eqref{limi} (\cite{Mladenovic}, Proposition 1.2.11).

We now analyze $Var[S_t]$.
The composition of  a function  $R_1\in RV_{\rho_1}$ with a function $R_2\in RV_{\rho_2}$ with  $\lim_{x\rightarrow \infty} R_2(x)=+\infty$, satisfies $R_1\circ R_2\in RV_{\rho_1\cdot \rho_2}$.
Thus, since
$\mu\in RV_{\frac{1}{p+2}}$ and $\phi\in RV_{-p-1}$, we have
\begin{equation}\label{pim}g_1(t)=\phi(\mu(t))\in RV_{-\frac{p+1}{p+2}}. 
\end{equation}
Therefore, using \cite{Mladenovic} Th. 1.2.1(a),
$$g_2(w)=\int_w^\infty \phi(x)dx \in RV_{-p}.$$
Since $\phi\in RV_{-p-1}$, we have $\frac{1}{\phi(w)^2}\in RV_{2(p+1)}$. Therefore $$g_3(w)=\frac{1}{\phi(w)^2}\cdot g_2(w)\in RV_{p+2}.$$
Therefore, again using \cite{Mladenovic} Th. 1.2.1(a),
$$g_4(v)=\int_0^v g_3(w)dw=\int_0^v \frac{1}{\phi(w)^2}\int_w^\infty \phi(x)dx dw\in RV_{p+3},$$
which implies
$$g_5(t)=g_4(\mu(t))=\int_0^{\mu(t)} \frac{1}{\phi(w)^2}\int_w^\infty \phi(x)dx dw\in RV_{\frac{p+3}{p+2}}.$$ 
Together with 
\eqref{pim} we conclude that
$$g_6(t)=g_1(t)g_5(t)=\phi(\mu(t))\int_0^{\mu(t)} \frac{1}{\phi(w)^2}\int_w^\infty \phi(x)dx dw \in RV_{\frac{2}{p+2}},$$
and finally, using \eqref{varf},
$$Var[S_t]=\frac{Var[X_0]}{\phi(\mu_0)}\cdot g_1+g_5\in RV_{\frac{2}{p+2}}.$$

\end{proof}

Note that the Pareto distribution 
\eqref{pareto} 
satisfies \eqref{alg} with $p=\alpha-2$, so we conclude $\mu\in RV_{\frac{1}{\alpha}}$ and $Var[S_t]\in RV_{\frac{2}{\alpha}}$, as can also be seen directly from the explicit expressions \eqref{mu_pareto_g}, \eqref{var_pareto}.

As another example, consider the Fr\'echet distribution
\begin{equation}\label{Frechet}F(x)=e^{-x^{-\alpha}},\;\;\alpha>1.\end{equation}
We have, making repeated use of L'H\^opital's rule
\begin{align*}\lim_{x\rightarrow \infty} \frac{x\phi(x)}{\int_x^{\infty}\phi(u)du}
&=-\lim_{x\rightarrow \infty} \frac{\phi(x)+x\phi'(x)}{\phi(x)}
=-1-\lim_{x\rightarrow \infty} \frac{x\phi'(x)}{\phi(x)}=-1-\lim_{x\rightarrow \infty}\frac{\phi'(x)+x\phi''(x)}{\phi'(x)}\\&=-2+\lim_{x\rightarrow \infty}\frac{xF'(x)}{1-F(x)}=-2+\lim_{x\rightarrow \infty}\frac{\alpha x^{-\alpha}e^{-x^{-\alpha}}}{1-e^{-x^{-\alpha}}}=-2+\alpha,\end{align*}
so \eqref{alg} holds with $p=\alpha-2$, and we conclude from
Theorem \ref{pow} that $\mu(t)\in RV_{\frac{1}{\alpha}}$, and $Var[S_t]\in RV_{\frac{2}{\alpha}}$.

\subsection{Large $t$ limit of the distribution of $\hat{T}_t$}
\label{T_large}

In this subsection we are interested in
the existence and nature of a limiting distribution of the random variable $\hat{T}_t=\frac{1}{t}T_t$,
see \eqref{htt1},
\begin{equation}\label{hinf}\hat{H}_{\infty}(s)=\lim_{t\rightarrow\infty} \hat{H}_t(s),\;\;\;s\geq 0
\end{equation}
When this distribution exists, the random variables $\hat{T}_t$ converge in distribution to a random variable 
$\hat{T}_\infty$ with the CDF $\hat{H}_\infty$, which expresses the fraction of the marketing period which will be utilized when the marketing period is long.

In all three examples considered in subsection
\ref{examples_T} (see equations \eqref{limH1},\eqref{limH2},\eqref{limH3}) the limit \eqref{hinf} exists, with
\begin{equation}\label{limHg}\hat{H}_{\infty}(s)=\begin{cases}
1-(1-s)^{\gamma} & s<1\\
1 & s\geq 1
\end{cases},\end{equation}
where $\gamma=2$ in the case of a uniform distribution,
$\gamma=1$ in the case of an exponential distribution, and $\gamma=1-\frac{1}{\alpha}$ in the case of a
Pareto distribution with parameter $\alpha>1$. The next result shows that this is not a coincidence:

\begin{thm}\label{thmH}
The limiting distribution defined by \eqref{hinf} exists if and only if
$\mu'(t)$ is a regularly varying function of index $\rho\leq 0$, and in this case $\hat{H}_{\infty}(s)$ is of the form \eqref{limHg}, where 
$\gamma=-\rho$. In particular, in this case 
\begin{equation}\label{lgg}\lim_{t\rightarrow \infty}\frac{E[T_t]}{t}=\lim_{t\rightarrow \infty}E[\hat{T}_t]=\frac{1}{\gamma+1},\;\;\;\lim_{t\rightarrow \infty}\frac{Var[T_t]}{t^2}=\lim_{t\rightarrow \infty}Var[\hat{T}_t]=\frac{\gamma}{(\gamma+1)^2(\gamma+2)}.\end{equation}
\end{thm}

\begin{proof}
Using \eqref{ode} we can write \eqref{htt1} as
\begin{equation}\label{jj}\hat{H}_t(s)=
\begin{cases}
	1-\frac{\mu'(t)}{\mu'((1-s)t)} & s<1\\
	1 & s\geq 1
\end{cases}.
\end{equation}
so that the limit \eqref{hinf} exists if and only if the limit 
$$\lim_{t\rightarrow \infty}\frac{\mu'(xt)}{\mu'(t)}$$
exists for $x\in [0,1)$. By \cite{Mladenovic}, Theorems 1.1.8,1.1.9, existence of these limits 
is equivalent to $\mu'(t)$ being regularly varying with some index $\rho$, which means that, for all $x\geq0$,
$$\lim_{t\rightarrow \infty}\frac{\mu'(xt)}{\mu'(t)}=x^\rho,$$
and since $\mu$ is concave (Theorem \ref{char}) so that $\mu'(t)$ is decreasing, $\rho\leq 0$. Therefore
$$\lim_{t\rightarrow \infty}\frac{\mu'(t)}{\mu'((1-s)t)}=(1-s)^{-\rho},$$
which, in view of \eqref{jj}, implies 
\eqref{limHg} with $\gamma=-\rho$.
\eqref{lgg} is obtained by computing the expectation and variance of the 
distribution \eqref{limHg}.
\end{proof}

The above theorem ensures that if the limit \eqref{hinf} exists it must be of the form \eqref{limHg}, but the existence of the limit hinges on  $\mu'(t)$ being of regular variation. 
We now show that this is the case for 
several classes of offer distributions.

\subsubsection{Offer distributions with  support bounded from the right}

\begin{thm}
Assume $M<\infty$, where $M$ is given by \eqref{def_M}, and that $F$ satisfies \eqref{edge} for some $p>0,c>0$.
Then \eqref{limHg} holds with $\gamma=1+\frac{1}{p}$. In particular, we have
\begin{equation}\label{lg1}\lim_{t\rightarrow \infty}\frac{E[T_t]}{t}=\lim_{t\rightarrow \infty}E[\hat{T}_t]=\frac{p}{2p+1},\;\;\;\lim_{t\rightarrow \infty}\frac{Var[T_t]}{t^2}=\lim_{t\rightarrow \infty}Var[\hat{T}_t]=\frac{p^2(p+1)}{(2p+1)^2(3p+1)}.\end{equation}

\end{thm}
As an example, for the Beta distribution \eqref{Beta}, we have \eqref{edge} with 
$M=1,p=\beta,c=\frac{\Gamma(\alpha+\beta)}{\Gamma(\alpha)\Gamma(\beta+1)}$,
so \eqref{ba} gives \eqref{limHg} with $\gamma=\frac{1}{\beta}+1$.
In particular, when $\alpha=\beta=1$ we get the uniform distribution on $[0,1]$, and we recover \eqref{atu}, \eqref{limH1}.

\begin{proof}

Substituting $x=\mu(t)$ in \eqref{edge}
gives
\begin{equation}\label{lll1}
\lim_{t\rightarrow +\infty}\frac{\phi(\mu(t))}{(M-\mu(t))^{p+1}}=c.
\end{equation}
The assumption of the current 
theorem is identical to that of Theorem \ref{tep}, so that we have \eqref{ed3}, that is
\begin{equation}\label{bob1}\lim_{t\rightarrow +\infty}\frac{\phi(\mu(t))}{t^{-\frac{p+1}{p}}}=\frac{1}{(\lambda p)^{\frac{p+1}{p} }c^{\frac{1}{p}}}.\end{equation}
Substituting $(1-s)t$ instead of $t$
in \eqref{bob1} gives
\begin{equation}\label{bob2}\lim_{t\rightarrow +\infty}\frac{\phi(\mu((1-s)t))}{t^{-\frac{p+1}{p}}}=\frac{1}{(\lambda p)^{\frac{p+1}{p} }c^{\frac{1}{p}}}\cdot (1-s)^{-\frac{p+1}{p}}.\end{equation}
Combining \eqref{bob1} and \eqref{bob2} gives
$$\lim_{t\rightarrow +\infty}\frac{\phi(\mu(t))}{\phi(\mu((1-s)t))}=(1-s)^{\frac{p+1}{p}},$$
which, together with \eqref{htt1},
gives the result.
\end{proof}

\subsubsection{Offer distributions with 
exponential tails}

\begin{thm}
Assume \eqref{exptail} holds. Then 
\eqref{limHg} holds, with $\gamma=1$.
In particular, we have
\begin{equation}\label{lg2}\lim_{t\rightarrow \infty}\frac{E[T_t]}{t}=\lim_{t\rightarrow \infty}E[\hat{T}_t]=\frac{1}{2},\;\;\;\lim_{t\rightarrow \infty}\frac{Var[T_t]}{t^2}=\lim_{t\rightarrow \infty}Var[\hat{T}_t]=\frac{1}{12}.\end{equation}
\end{thm}

\begin{proof}
Using Theorem \ref{thm:exptail},
we have \eqref{expc}, so
$$\lim_{t\rightarrow\infty} t\mu'(t)=\frac{1}{c},\;\;\;\lim_{t\rightarrow\infty} (1-s)t\mu'((1-s)t)=\frac{1}{c}\;\;\;\Rightarrow\;\;\;\lim_{t\rightarrow\infty}\frac{\mu'(t)}{\mu'((1-s)t)}=1-s,$$
so \eqref{jj} gives \eqref{limHg} with 
$\gamma=1$.
\end{proof}

As an example, we have seen in 
subsection \ref{ofexp} that the Gamma distribution \eqref{gammad} satisfies \eqref{exptail}, so we conclude that 
\eqref{limHg} holds with $\gamma=1$.
In particular for the exponential distribution ($\alpha=1$ in \eqref{gammad}), we recover \eqref{texpa},\eqref{limH2}.

\subsubsection{Offer distributions with power-law tails}

\begin{thm}
Suppose the function $\phi(x)$ 
satisfies \eqref{alg}. Then \eqref{limHg} holds with $\gamma=\frac{p+1}{p+2}$.
In particular, we have
\begin{equation}\label{lg}\lim_{t\rightarrow \infty}\frac{E[T_t]}{t}=\lim_{t\rightarrow \infty}E[\hat{T}_t]=\frac{p+2}{2p+3},\;\;\;\lim_{t\rightarrow \infty}\frac{Var[T_t]}{t^2}=\lim_{t\rightarrow \infty}Var[\hat{T}_t]=\frac{(p+1)(p+2)^2}{(3p+5)(2p+3)^2}.\end{equation}
\end{thm}

\begin{proof}
By \eqref{ode} and \eqref{pim} we have 
$\mu'(t)=\lambda \phi(\mu(t))\in RV_{-\frac{p+1}{p+2}}$. By Theorem 
\ref{thmH}, this implies \eqref{limHg}, with $\gamma=\frac{p+1}{p+2}$.
\end{proof}
The Pareto distribution 
\eqref{pareto} 
satisfies \eqref{alg} with $p=\alpha-2$, so we conclude 
\eqref{limHg} holds with 
$\gamma=\frac{\alpha-1}{\alpha}$, recovering \eqref{as_pareto},\eqref{limH3}.

As another example, for the Fr\'echet 
distribution \eqref{Frechet}, we have
\eqref{alg} with $p=\alpha-2$, so that
\eqref{limHg} holds with 
$\gamma=\frac{\alpha-1}{\alpha}$.


\begin{thebibliography}{9}
	
	\bibitem{Allaart}
	Allaart, P. C. (2007). Prophet inequalities for iid random variables with random arrival times. Sequential Analysis, 26(4), 403-413.
	
	\bibitem{Bauerle}
	B\"auerle, N., \& Rieder, U. (2011). Markov decision processes with applications to finance. Springer Science \& Business Media.
	
	\bibitem{Cayley}
	A. Cayley (1875) “Mathematical questions with their solutions”, The Collected Mathematical Papers of Arthur Cayley Vol X (1896) Cambridge Univ. Press, 587-588.
	
	\bibitem{Correa}
	Correa, J., Foncea, P., Hoeksma, R., Oosterwijk, T., \& Vredeveld, T. (2019). Recent developments in prophet inequalities. ACM SIGecom Exchanges, 17(1), 61-70.
	
	\bibitem{David}
	David, I., \& Yechiali, U. (1985). A time-dependent stopping problem with application to live organ transplants. Operations Research, 33(3), 491-504.
	
	\bibitem{DeGroot}
	DeGroot, M. H. (2005). Optimal statistical decisions. John Wiley \& Sons.
	
	\bibitem{Elfving} Elfving, G. (1967). A persistency problem connected with a point process. Journal of Applied Probability, 4(1), 77-89.
	
	\bibitem{entwistle}
	Entwistle, H. N., Lustri, C. J., \& Sofronov, G. Y. (2022). On asymptotics of optimal stopping times. Mathematics, 10(2), 194.
	
	\bibitem{Ferguson}
	Ferguson, T.S. (2006). Optimal stopping and applications. http://www.math.ucla.edu/~tom/Stopping/Contents.html.
	
	\bibitem{fitch} Finch, S., 2024. A deceptively simple quadratic recurrence.  arXiv:2409.03510.
	
	\bibitem{Gilbert}
	Gilbert, J. P., \& Mosteller, F. (1966). Recognizing the maximum of a sequence. Journal of the American Statistical Association, 61(313), 35-73.
	
	\bibitem{Guttman}
	Guttman, I. (1960). On a problem of L. Moser. Canadian Mathematical Bulletin, 3(1), 35-39.
	
	\bibitem{Hayes}
	Hayes, R. H. (1969). Optimal strategies for divestiture. Operations Research, 17(2), 292-310.
	
	\bibitem{Hill}
	Hill, T. P., \& Kertz, R. P. (1992). A survey of prophet inequalities in optimal stopping theory. Contemporary Mathematics, 125(1), 191.
	
	\bibitem{Karlin}
	Karlin, S. (1962). Stochastic models and optimal policy for selling an asset, in
	Arrow, K. J., Karlin, S., \& Scarf, H., Studies in applied probability and management science, Stanford University Press.
	
	\bibitem{Kennedy}
	Kennedy, D. P., \& Kertz, R. P. (1991). The asymptotic behavior of the reward sequence in the optimal stopping of iid random variables. The Annals of Probability, 19(1), 329-341.
	
	\bibitem{Livanos}
	Livanos, V., \& Mehta, R. (2025, July). Minimization IID Prophet Inequality via Extreme Value Theory: A Unified Approach. In Proceedings of the 26th ACM Conference on Economics and Computation (pp. 1157-1179).
	
	\bibitem{Lucier}
	Lucier, B. (2017). An economic view of prophet inequalities. ACM SIGecom Exchanges, 16(1), 24-47.
	
	\bibitem{Mladenovic}
	Mladenovi\'c, P. (2024). Extreme Values in Random Sequences. Cham, Switzerland: Springer.
	
	\bibitem{Mazalov}
	Mazalov, V. V., \& Peshkov, N. V. (2004). On asymptotic properties of optimal stopping time. Theory of Probability \& Its Applications, 48(3), 549-555.
	
	\bibitem{Moser}
	Moser, L. (1956). On a Problem of Cayley. Scripta Mathematica, 22, 289–292.
	
	\bibitem{Sakaguichi}
	Sakaguchi, M. (1976). Optimal stopping problems for randomly arriving offers. Mathematicae Japonicae, 21, 201-217.
	
	
	
\end{thebibliography}
\end{document}